\documentclass[11pt]{article}

\usepackage[a4paper,margin=27mm]{geometry}
\usepackage[T1]{fontenc}
\usepackage{lmodern}
\usepackage[english]{babel}
\usepackage{amsmath,amssymb,amsthm,mathtools,mathrsfs}
\usepackage{aliascnt}
\usepackage{enumitem,needspace}
\usepackage{microtype}
\usepackage{xcolor}
\usepackage[colorlinks=true,
 linkcolor=blue!55!black,
 citecolor=blue!55!black,
 urlcolor=blue!55!black,
 pdftitle={Reconstruction of Heat Geometry from a Scalar Process},
 pdfauthor={Obayda Julien Assaad},
 pdfsubject={Gaussian processes, tensor invariants, and reconstruction of heat operators and geometry}]{hyperref}
\usepackage[nameinlink,capitalise,noabbrev]{cleveref}

\allowdisplaybreaks
\setlist{itemsep=0.35em,topsep=0.55em}

\newtheorem{theorem}{Theorem}[section]

\newaliascnt{proposition}{theorem}
\newtheorem{proposition}[proposition]{Proposition}
\aliascntresetthe{proposition}

\newaliascnt{lemma}{theorem}
\newtheorem{lemma}[lemma]{Lemma}
\aliascntresetthe{lemma}

\newaliascnt{corollary}{theorem}
\newtheorem{corollary}[corollary]{Corollary}
\aliascntresetthe{corollary}

\newaliascnt{definition}{theorem}
\newtheorem{definition}[definition]{Definition}
\aliascntresetthe{definition}

\newaliascnt{remark}{theorem}
\newtheorem{remark}[remark]{Remark}
\aliascntresetthe{remark}

\newtheorem{introthm}{Theorem}

\newaliascnt{introcor}{introthm}
\newtheorem{introcor}[introcor]{Corollary}
\aliascntresetthe{introcor}

\crefname{introthm}{Theorem}{Theorems}
\crefname{introcor}{Corollary}{Corollaries}

\newcommand{\E}{\mathbb E}
\newcommand{\R}{\mathbb R}
\newcommand{\Corr}{\operatorname{Corr}}
\newcommand{\Tr}{\operatorname{Tr}}
\newcommand{\Ran}{\operatorname{Ran}}
\newcommand{\Span}{\operatorname{span}}
\newcommand{\Sym}{\operatorname{Sym}}
\newcommand{\Id}{\operatorname{Id}}
\newcommand{\norm}[1]{\left\lVert #1\right\rVert}
\newcommand{\ip}[2]{\left\langle #1,#2\right\rangle}
\newcommand{\Rep}{\operatorname{Rep}}
\newcommand{\Law}{\operatorname{Law}}
\newcommand{\Stab}{\operatorname{Stab}}

\title{Reconstruction of Heat Geometry from a Scalar Process}
\author{Obayda Julien Assaad\\
\small Independent researcher, Chi\c{s}in\u{a}u, Moldova}
\date{}

\begin{document}
\maketitle

\begin{abstract}
A smooth scalar process with one fixed Gaussian covariance kernel determines
any finite graded family of symmetric Hilbert tensors up to a common
orthogonal isomorphism of their minimal supports. For two heat tensors,
its law determines a closed
positive-dimensional Riemannian manifold, a Euclidean bundle, its connection
and potential, and the unknown heat time, up to isometry and gauge.
The quadratic tensor recovers the scaled generator; the quartic tensor
recovers the spatial norm and volume, which fix the time scale. For compact
finite-dimensional RCD spaces, the same law determines the measured space
with distance divided by the square root of the heat time.
\end{abstract}



\section{Introduction}
\label{sec:introduction}

The spectrum of a Laplace operator need not determine its geometry
\cite{Kac1966,Milnor1964,Vigneras1980,Sunada1985,GordonWebbWolpert1992}.
An intertwiner of Laplacians becomes geometric when it also preserves
the spatial structure. Order isomorphisms provide one such condition
\cite{ArendtBiegertTerElst2012,LenzSchmidtWirth2019}.
Here that structure is encoded by two tensors and recovered from the law
of one scalar process. The heat time is also recovered when the measure
is the Riemannian volume and the dimension is positive.

Every process uses the covariance kernel
\begin{equation}
 K(u,v)=e^{-(u-v)^2},\qquad u,v\in\mathbb R.
 \label{eq:intro-canvas-kernel}
\end{equation}
The datum is its law on $C(\mathbb R)$, equivalently all finite-dimensional
distributions. A one-point marginal does not suffice
(\cref{sec:hidden-Hermite-rank}). The parameter $u$ labels correlated
Gaussian observations; it is distinct from the heat time $t$.

\subsection{Reconstruction of symmetric tensors}

Let $H$ be a real separable Hilbert space and
\begin{equation}
 \mathbf f=(f_0,f_1,\ldots,f_m),\qquad
 f_0\in\mathbb R,\quad f_q\in H^{\odot q}.
 \label{eq:intro-tensor-packet}
\end{equation}
For jointly isonormal processes satisfying
$\E[\mathbb W_u(h)\mathbb W_v(k)]=K(u,v)\langle h,k\rangle_H$, put
\[
 \mathcal Z_{\mathbf f}(u)=f_0+\sum_{q=1}^m I_q^{\mathbb W_u}(f_q),
 \qquad \E[I_q(f)I_q(g)]=q!\langle f,g\rangle.
\]
The process has a smooth modification. Let $S(\mathbf f)$ be the smallest
closed $S\subset H$ with $f_q\in S^{\odot q}$ for every $q\ge1$.

\begin{introthm}[Fixed-kernel reconstruction;
\cref{thm:universal-canvas-replica-tannaka}]
\label{thm:intro-replica-tannaka}
Two families $\mathbf f,\mathbf f'$ have the same process law if and only
if $f_0=f'_0$ and an orthogonal isomorphism
$U:S(\mathbf f)\to S(\mathbf f')$ satisfies
$f'_q=U^{\odot q}f_q$ for every $q\ge1$.
Zero tensors may be added to equalize the maximal degrees.
\end{introthm}

\subsection{Geometry and an unknown heat time}
\label{subsec:intro-heat-tensors}

A \emph{heat geometry} is
\[
 \mathfrak X=(M,g,E,\nabla^E,V_E),\qquad
 \mathcal L_{\mathfrak X}=(\nabla^E)^*\nabla^E+V_E,
\]
where $M$ is closed and connected, $E$ is a positive-rank Euclidean
bundle, $\nabla^E$ is a metric connection, and $V_E$ is smooth and
self-adjoint. Write $H_{\mathfrak X}=L^2(M,E;dv_g)$.
The tensors $Q_{\mathfrak X,t}\in H_{\mathfrak X}^{\odot2}$ and
$C_{\mathfrak X,t}\in H_{\mathfrak X}^{\odot4}$ are defined by
\begin{align*}
 Q_{\mathfrak X,t}^{\flat}&=e^{-t\mathcal L_{\mathfrak X}},\\
 \langle C_{\mathfrak X,t},a^{\otimes4}\rangle
 &=\|e^{-t\mathcal L_{\mathfrak X}/2}a\|_{L^4(M,E)}^4.
\end{align*}
Here $\langle Q^\flat a,b\rangle=\langle Q,a\otimes b\rangle$.
Heat smoothing makes both tensors Hilbert tensors. Set
\[
 \mathcal Z_{\mathfrak X,t}(u)
 =I_2^{\mathbb W_u}(Q_{\mathfrak X,t})
  +I_4^{\mathbb W_u}(C_{\mathfrak X,t}).
\]

\begin{introthm}[Unknown-time heat rigidity;
\cref{thm:scalar-observable-hears-gauge-field}]
\label{thm:intro-full-faithfulness}
Let $\mathfrak X,\mathfrak Y$ have positive-dimensional bases and let
$t,s>0$. The following are equivalent:
\begin{enumerate}[label=\textup{(\roman*)}]
\item $\mathcal Z_{\mathfrak X,t}$ and $\mathcal Z_{\mathfrak Y,s}$
have the same law on $C(\mathbb R)$;
\item an orthogonal isomorphism transports
$(Q_{\mathfrak X,t},C_{\mathfrak X,t})$ to
$(Q_{\mathfrak Y,s},C_{\mathfrak Y,s})$;
\item $t=s$ and a Riemannian isometry with a smooth orthogonal bundle
isomorphism transports the connection and potential.
\end{enumerate}
Every operator in \textup{(ii)} is induced by a unique geometric pair
in \textup{(iii)}. Thus the tensor stabilizer is the geometric
automorphism group.
\end{introthm}

No common dimension, common rank, or simple-spectrum assumption is needed.
For a trivial line bundle and zero potential, the law determines the
manifold and the heat time. With a prescribed common time, the result
also holds in dimension zero.

\subsection{Proof strategy}

Two reversible operations organize the proof.

First, the process law gives tensor contractions. Wick's formula writes
each mixed moment as a finite sum of
\[
 \exp\!\left(-\sum_{a<b}N_{ab}(u_a-u_b)^2\right)
\]
with loop-free contractions as coefficients. Distinct edge
multiplicities give linearly independent functions, so the coefficients
are determined. Define the positive trace-class operator
\[
 D_{\mathbf f}=\sum_{q=1}^m L_{f_q}L_{f_q}^*,
 \qquad L_{f_q}:H^{\otimes(q-1)}\to H.
\]
Its support is $S(\mathbf f)$. Spectral projections of $D_{\mathbf f}$
are trace-norm limits of polynomials in $D_{\mathbf f}$ with zero
constant term. In a contraction, these polynomials replace a
self-pairing by pairings between distinct tensor copies. Thus the
observed loop-free contractions determine all contractions on each
finite spectral subspace. Orthogonal invariant theory reconstructs
each truncation; compactness makes their isomorphisms compatible.

Second, remove the heat regularization. With $Q=Q_{\mathfrak X,t}$,
$C=C_{\mathfrak X,t}$ and $\mathcal R=\Ran((Q^\flat)^{1/2})$, the pair
determines
\begin{equation}
 B=-\log Q^\flat=t\mathcal L_{\mathfrak X},\qquad
 \|v\|_4^4
 =\big\langle C,\big((Q^\flat)^{-1/2}v\big)^{\otimes4}\big\rangle,
 \quad v\in\mathcal R.
 \label{eq:intro-remove-smoothing}
\end{equation}
The heat range is dense in $L^4$. A tensor isomorphism therefore
preserves both $L^2$ and $L^4$, hence the base measure and the fiber
metric. Intertwining $B$ makes its measurable base and fiber maps smooth.
The principal symbol identifies the metric up to the ratio $s/t$:
\[
 h=(s/t)\phi^*g,
 \qquad dv_h=\phi^*dv_g.
\]
Since $dv_h=(s/t)^{n/2}\phi^*dv_g$ and $n>0$, one has $s=t$.
The first- and zeroth-order terms then identify the connection and
potential. This argument uses no heat-trace asymptotic.

The ingredients---Wick contractions, orthogonal invariant theory,
Lamperti rigidity and elliptic regularity---are classical
\cite{Isserlis1918,JansonGaussian,PeccatiTaqqu2011,WeylClassicalGroups,
Brauer1937,Schrijver2008,YoungOrthogonalHolant,Sourour1978,GreimJamison1987}.
The construction combines them to recover infinite-dimensional tensor
orbits from one fixed covariance kernel and to recover geometry and time
from two heat tensors.

\subsection{Relation to other reconstruction results}

Graph partition functions provide related invariant-theoretic
reconstructions \cite{FreedmanLovaszSchrijver2007,LovaszSzegedy2009}.
The finite-dimensional result of \cite{AssaadFinite2026} uses prescribed
correlated Gaussian experiments and gives bi-H\"older comparison with
orbit distance on coefficient balls. Minimal supports and homogeneous
realizations also appear in \cite{AssaadPrimitiveFock2026};
\cite{HerryMalicetPoly2024} concerns weak limits.

Other geometric data include Wick squares of Gaussian free fields
\cite{Dang2022}, eigenfunction triple products and local Weyl data
\cite{SchaeferTripleProducts,WangWymanXi2024}. Local hyperbolic boundary
measurements recover manifolds and connections
\cite{KurylevOksanenPaternain2018}; Bochner spectra recover connections
under further geometric and rank assumptions
\cite{CekicLefeuvre2025}. Here the quartic tensor forces the heat
intertwiner to preserve the spatial norm and its measure scale.

\subsection{Dirichlet spaces and RCD spaces}

For irreducible symmetric Dirichlet forms, the same pair determines the
measure algebra and the form multiplied by the heat time, under the
heat-range hypotheses of \cref{sec:dirichlet-rcd}. Compact
finite-dimensional RCD spaces satisfy those hypotheses. Their
Sobolev-to-Lipschitz property recovers the distance
\cite{AmbrosioGigliSavare2014Invent,AmbrosioGigliSavare2014Duke,
ErbarKuwadaSturm2015,JiangLiZhang2016,Gigli2018}.

\begin{introcor}[Compact RCD spaces; \cref{cor:canvas-hears-compact-RCD}]
\label{thm:intro-rcd-reconstruction}
Let $(X_i,d_i,\mu_i)$ be compact connected
$\operatorname{RCD}(K_i,N_i)$ spaces with full support and finite
synthetic dimension, and let $t_i>0$. Their quadratic--quartic
process laws agree if and only if
\[
 (X_1,t_1^{-1/2}d_1,\mu_1)
 \simeq (X_2,t_2^{-1/2}d_2,\mu_2)
\]
by a measure-preserving isometry.
\end{introcor}

At a known common time this recovers the original measured spaces.
For arbitrary measures, rescaling $d$ to $cd$ and $t$ to $c^2t$
leaves the data unchanged. The unknown-time conclusion of Theorem~B
uses the unnormalised Riemannian volume.

\subsection{Organization}

Sections~\ref{sec:replica-register}--\ref{subsec:universal-gaussian-canvas}
prove tensor reconstruction. Section~\ref{sec:gauge-field} recovers
geometry and time; Section~\ref{sec:dirichlet-rcd} treats Dirichlet forms
and RCD spaces. The appendices describe the tensor categories and
stabilizers, homogeneous-chaos realizations, instability and inverse
continuity, finite statistical tests, the insufficiency of marginals,
and a dilation characterization of the Gaussian covariance kernel.

\section{Correlated Gaussian observations and tensor contractions}
\label{sec:replica-register}

Let $H$ be a real separable Hilbert space. We use $H^{\otimes0}=\mathbb R$ and
\[
 \Sym_q(h_1\otimes\cdots\otimes h_q)
 =\frac1{q!}\sum_{\pi\in\mathfrak S_q}
    h_{\pi(1)}\otimes\cdots\otimes h_{\pi(q)},\qquad
 H^{\odot q}=\Ran\Sym_q.
\]
For $\xi\in H^{\odot p}$ and $\eta\in H^{\odot q}$, set
$\xi\odot\eta=\Sym_{p+q}(\xi\otimes\eta)$.
An isonormal Gaussian process
over $H$ is a centered Gaussian linear map $W:H\to L^2(\Omega)$ with
$\E[W(h)W(k)]=\ip hk_H$. The Hermite polynomials and Wiener integrals
are normalized by
\[
 e^{zx-z^2/2}=\sum_{q\ge0}\operatorname{He}_q(x)\frac{z^q}{q!},
 \qquad
 I_q^W(h^{\otimes q})
 =\|h\|^q\operatorname{He}_q\!\left(\frac{W(h)}{\|h\|}\right),\quad h\ne0.
\]
Polarization and continuous linear extension define $I_q^W$ on
$H^{\odot q}$, with
\[
 \E[I_q^W(f)I_p^W(g)]
 =\mathbf1_{\{p=q\}}q!\ip fg_{H^{\otimes q}}.
\]
For a bounded operator $A$, write $A^{\odot q}=A^{\otimes q}|_{H^{\odot q}}$.
Fix
\begin{equation}
 \mathbf f=(f_0,f_1,\ldots,f_m),
 \qquad f_0\in\R,\quad f_q\in H^{\odot q},
 \label{eq:finite-tensor-packet}
\end{equation}
and put
\begin{equation}
 F_{\mathbf f}(W)=f_0+\sum_{q=1}^m I_q^W(f_q).
 \label{eq:finite-wiener-polynomial}
\end{equation}
For $\Sigma\in\Corr_r^\circ
=\{\Sigma=\Sigma^{\mathsf T}>0:\Sigma_{aa}=1\}$, choose a jointly
Gaussian family $(W_a)_{a=1}^r$ with
$\E[W_a(h)W_b(k)]=\Sigma_{ab}\ip hk_H$.
Explicitly, for independent isonormal processes $W^{(1)},\ldots,W^{(r)}$,
\[
 W_a(h)=\sum_{b=1}^r(\Sigma^{1/2})_{ab}W^{(b)}(h).
\]

\begin{definition}[Joint laws under Gaussian correlations]
The correlated observation data of $\mathbf f$ are
\begin{equation}
 \begin{split}
 \Rep(\mathbf f)
 &=\bigl(\Phi_{\mathbf f}(r,\Sigma,z)\bigr)_
 {r\ge1,\,\Sigma\in\Corr_r^\circ,\,z\in\R^r},\\
 \Phi_{\mathbf f}(r,\Sigma,z)
 &=\E\exp\!\left(i\sum_{a=1}^rz_aF_{\mathbf f}(W_a)\right).
 \end{split}
 \label{eq:scalar-replica-signature}
\end{equation}
\end{definition}

For $p\ge2$, hypercontractivity \cite{JansonGaussian} gives
\[
 \|F_{\mathbf f}(W)\|_p
 \le |f_0|+\sum_{q=1}^m(p-1)^{q/2}\sqrt{q!}\,\|f_q\|_2.
\]
Thus
\[
 M_{\mathbf f,r}(\Sigma):=\E\prod_{a=1}^rF_{\mathbf f}(W_a)
 =i^{-r}\left.\partial_{z_1}\cdots\partial_{z_r}
               \Phi_{\mathbf f}(r,\Sigma,z)\right|_{z=0}.
\]

\subsection{Contractions and their moment coefficients}

Fix an orthonormal basis $(e_i)_{i\in I}$ of $H$, use $I^0=\{()\}$, and write
$f_{i_1\ldots i_p}=\langle f,e_{i_1}\otimes\cdots\otimes e_{i_p}\rangle$.
For $f\in H^{\otimes p}$, $g\in H^{\otimes q}$ and
$0\le k\le\min(p,q)$, define
\begin{equation}
 (f\otimes_k g)_{\alpha\beta}
 :=\sum_{\gamma\in I^k}f_{\alpha\gamma}g_{\gamma\beta},
 \qquad \alpha\in I^{p-k},\quad\beta\in I^{q-k}.
 \label{eq:partial-contraction-definition}
\end{equation}
Cauchy--Schwarz gives
\begin{equation}
 \begin{split}
 \|f\otimes_k g\|_2^2
 &=\sum_{\alpha,\beta}\left|\sum_\gamma
                    f_{\alpha\gamma}g_{\gamma\beta}\right|^2\\
 &\le\sum_{\alpha,\beta}
       \left(\sum_\gamma|f_{\alpha\gamma}|^2\right)
       \left(\sum_\gamma|g_{\gamma\beta}|^2\right)
 =\|f\|_2^2\|g\|_2^2.
 \end{split}
 \label{eq:partial-contraction-bound}
\end{equation}

Let $N=(N_{ab})_{a,b=1}^r$ be symmetric with nonnegative integral entries and
\begin{equation}
 N_{aa}=0,\qquad q_a=\sum_{b\ne a}N_{ab}\in\{1,\ldots,m\}.
 \label{eq:loop-free-adjacency}
\end{equation}
For each $a<b$, introduce $N_{ab}$ indices
$i_{ab,1},\ldots,i_{ab,N_{ab}}\in I$, and set $i_{ba,\ell}=i_{ab,\ell}$.
Let $\mathbf i_a=(i_{ab,\ell})_{b\ne a,\,1\le\ell\le N_{ab}}$.
For symmetric $g_a\in H^{\odot q_a}$, define
\begin{equation}
 \operatorname{Contr}_N(g_1,\ldots,g_r)
 :=\sum_{(i_{ab,\ell})_{a<b}}\prod_{a=1}^r(g_a)_{\mathbf i_a}.
 \label{eq:complete-contraction-definition}
\end{equation}
Each index occurs in exactly two factors. For example,
\[
 \begin{aligned}
 N_{12}=q:&\qquad
 \operatorname{Contr}_N(f,g)=\sum_{i_1,\ldots,i_q}
 f_{i_1\ldots i_q}g_{i_1\ldots i_q}=\langle f,g\rangle,\\
 N_{12}=N_{23}=N_{13}=1:&\qquad
 \operatorname{Contr}_N(f,g,h)=\sum_{i,j,k}f_{ij}g_{jk}h_{ki}.
 \end{aligned}
\]
If two factors share $k$ indices, permute their slots to put those
indices last in the first factor and first in the second. Summing over
them replaces the two factors by $f\otimes_k g$. Repeated application of
\eqref{eq:partial-contraction-bound}, also to the absolute values of
the coefficients, gives
\begin{equation}
 \sum_{(i_{ab,\ell})_{a<b}}
       \prod_a|(g_a)_{\mathbf i_a}|
 \le\prod_a\|g_a\|_2.
 \label{eq:complete-contraction-bound}
\end{equation}
The sum is therefore absolutely convergent. For pure powers,
\[
 \operatorname{Contr}_N(h_1^{\otimes q_1},\ldots,h_r^{\otimes q_r})
 =\prod_{a<b}\langle h_a,h_b\rangle^{N_{ab}}.
\]
Pure powers span a dense subspace of each $H^{\odot q_a}$, so
\eqref{eq:complete-contraction-bound} proves basis independence.
The condition $N_{aa}=0$ excludes an index repeated inside one factor;
these are the \emph{loop-free} contractions. Set
$\Sigma^N=\prod_{a<b}\Sigma_{ab}^{N_{ab}}$.
For $P(\Sigma)=\sum_Nc_N\Sigma^N$, write $[\Sigma^N]P=c_N$.

\begin{proposition}[Contractions as moment coefficients]
\label{prop:exact-coefficient-formula}
For every $N$ satisfying \eqref{eq:loop-free-adjacency},
\begin{equation}
 [\Sigma^N]M_{\mathbf f,r}(\Sigma)
 =\frac{\prod_aq_a!}{\prod_{a<b}N_{ab}!}
 \operatorname{Contr}_N(f_{q_1},\ldots,f_{q_r}).
 \label{eq:exact-scalar-coefficient}
\end{equation}
Consequently, $\Rep(\mathbf f)$ determines all complete contractions
whose pairings are between distinct tensor copies.
\end{proposition}

\begin{proof}
For $h_1,\ldots,h_r\in H$, the Gaussian exponential identity is
\[
 \E\prod_{a=1}^r
 e^{z_aW_a(h_a)-z_a^2\|h_a\|^2/2}
 =\exp\!\left(\sum_{a<b}\Sigma_{ab}z_az_b\langle h_a,h_b\rangle\right).
\]
Equating coefficients of $\prod_a z_a^{q_a}$ gives
\[
 \E\prod_a I_{q_a}^{W_a}(h_a^{\otimes q_a})
 =\sum_{N:\,\sum_{b\ne a}N_{ab}=q_a}
 \frac{\prod_aq_a!}{\prod_{a<b}N_{ab}!}
 \Sigma^N\prod_{a<b}\langle h_a,h_b\rangle^{N_{ab}}.
\]
Polarization and continuity replace the pure powers by arbitrary
symmetric tensors \cite{JansonGaussian}. Expanding
$\prod_aF_{\mathbf f}(W_a)$ then gives \eqref{eq:exact-scalar-coefficient}.
The off-diagonal entries of $\Sigma$ vary in a neighbourhood of zero,
so the moment polynomial determines every coefficient. Finally,
$f_0=\E F_{\mathbf f}(W)$.
\end{proof}

\section{Orthogonal reconstruction in a separable Hilbert space}
\label{sec:tensor-reconstruction}

For a compact operator $A$ with singular values $s_j(A)$, put
\[
 \|A\|_{\mathcal S_p}=\left(\sum_j s_j(A)^p\right)^{1/p},
 \qquad \mathcal S_p=\{A:\|A\|_{\mathcal S_p}<\infty\},\quad p=1,2.
\]
For $q\ge1$, define
\[
 L_{f_q}:H^{\otimes(q-1)}\to H,\qquad
 \ip{L_{f_q}u}{v}=\ip{f_q}{v\otimes u}.
\]
In the basis used above,
\[
 (L_{f_q}u)_i=\sum_{\alpha\in I^{q-1}}(f_q)_{i\alpha}u_\alpha,
 \qquad
 \|L_{f_q}\|_{\mathcal S_2}^2
 =\sum_{i,\alpha}|(f_q)_{i\alpha}|^2=\|f_q\|_2^2.
\]
For $q=1$, $L_{f_1}(c)=cf_1$.
For $T\in H^{\odot2}$, the notation $T^\flat=L_T$ means
$(T^\flat)_{ij}=T_{ij}$; symmetry of $T$ gives $(T^\flat)^*=T^\flat$.

\begin{definition}[Minimal supporting subspace]
\[
 S(\mathbf f)=\overline{\Span}
 \{\Ran L_{f_q}:1\le q\le m\}\subset H.
\]
\end{definition}

\begin{lemma}[Characterization of the minimal support]
\label{lem:minimal-support}
The space $S(\mathbf f)$ is the smallest closed $S\subset H$ such that
$f_q\in S^{\odot q}$ for every $q\ge1$. Moreover,
$F_{\mathbf f}(W)$ depends only on $W|_{S(\mathbf f)}$.
\end{lemma}

\begin{proof}
Let $P$ be the orthogonal projection onto $S(\mathbf f)$. Then
$PL_{f_q}=L_{f_q}$, so
$(P\otimes\Id^{\otimes(q-1)})f_q=f_q$.
Symmetry gives the same identity in each slot; hence
$P^{\otimes q}f_q=f_q$. Conversely, $f_q\in S^{\odot q}$ implies
$L_{f_q}u\in S$ for every $u$.
Finally, $I_q^W(f_q)=I_q^{W|_{S(\mathbf f)}}(f_q)$.
\end{proof}

\begin{theorem}[Orbit determination from correlated observations]
\label{thm:scalar-replica-tannaka}
For finite symmetric tensor families $\mathbf f$ on $H$ and
$\mathbf f'$ on $H'$, the following are equivalent:
\begin{enumerate}[label=\textup{(\roman*)}]
\item $\Rep(\mathbf f)=\Rep(\mathbf f')$;
\item $f_0=f'_0$ and all corresponding loop-free contractions agree;
\item $f_0=f'_0$ and there is an orthogonal isomorphism
$U:S(\mathbf f)\to S(\mathbf f')$ with
$U^{\odot q}f_q=f'_q$ for every $q\ge1$.
\end{enumerate}
If $U_0$ is one such isomorphism, all choices are $U_0V$, where
$V\in O(S(\mathbf f))$ fixes every $f_q$.
\end{theorem}

\paragraph{Proof.}
The implication (i)$\Rightarrow$(ii) is
\cref{prop:exact-coefficient-formula}. For (ii)$\Rightarrow$(iii),
replace $H,H'$ by their minimal supports. If one support is zero,
the two-copy contractions $\|f_q\|^2=\|f'_q\|^2$ show that both are
zero. Otherwise use the following three steps.

\subsection{A positive trace-class operator}
\label{subsec:trace-class-operator}

Set
\[
 D=D_{\mathbf f}=\sum_{q=1}^mL_{f_q}L_{f_q}^*,\qquad
 D'=D_{\mathbf f'}=\sum_{q=1}^mL_{f'_q}L_{f'_q}^*.
\]
\[
 \Tr D=\sum_q\|f_q\|_2^2,\qquad
 \ker D=S(\mathbf f)^\perp=\{0\}.
\]
Its matrix coefficients and power traces are
\begin{align}
 D_{ij}&=\sum_{q=1}^m\sum_{\alpha\in I^{q-1}}
             (f_q)_{i\alpha}(f_q)_{j\alpha},
 \label{eq:support-operator-coefficients}\\
 \Tr D^k
 &=\sum_{q_1,\ldots,q_k=1}^m
   \sum_{i_1,\ldots,i_k,\,\alpha_1,\ldots,\alpha_k}
   \prod_{a=1}^k(f_{q_a})_{i_a\alpha_a}
                  (f_{q_a})_{i_{a+1}\alpha_a},
 \label{eq:support-power-contractions}
\end{align}
where $i_{k+1}=i_1$ and $\alpha_a\in I^{q_a-1}$.
Each inner sum in \eqref{eq:support-power-contractions} is a loop-free
contraction of $2k$ tensor copies. Thus $\Tr D^k=\Tr(D')^k$ for $k\ge1$.
Let $m_D(\lambda)=\dim\ker(D-\lambda I)$. The finite measures
\[
 \mu_D=\sum_{\lambda>0}m_D(\lambda)\lambda\,\delta_\lambda,
 \qquad
 \int x^j\,d\mu_D(x)=\Tr D^{j+1}
\]
are supported in $[0,\max(\|D\|,\|D'\|)]$ and have equal moments. Polynomial density
gives $\mu_D=\mu_{D'}$; hence $D,D'$ have the same positive eigenvalues
with the same multiplicities.

\subsection{Contractions of spectral truncations}
\label{subsec:spectral-truncations}

\begin{lemma}[Contractions with trace-class operators]
\label{lem:trace-weighted-network-bound}
Let $G=(V,E)$ specify ordered pairs of slots of tensors
$\tau_v\in H^{\otimes d_v}$, each slot occurring in exactly one pair.
A pair may join two slots of the same tensor. For
$T_e\in\mathcal S_1(H)$, use the bilinear form
$\langle T_ex,y\rangle$ on pair $e$. The resulting multilinear form
$Z_G$ is well defined and satisfies
\[
 |Z_G((\tau_v),(T_e))|
 \le\prod_v\|\tau_v\|_2\prod_e\|T_e\|_1.
\]
It is jointly continuous in the Hilbert norms of the tensors and the
trace norms of the operators.
\end{lemma}

\begin{proof}
Choose singular-value decompositions with orthonormal singular vectors:
\[
 T_ex=\sum_j s_{e,j}\langle v_{e,j},x\rangle u_{e,j},
 \qquad s_{e,j}\ge0,\quad
 \sum_j s_{e,j}=\|T_e\|_1.
\]
For $e=((v,a),(w,b))$ and a family of indices $(j_e)$, set
\[
 \xi_{v,a}=v_{e,j_e},\qquad \xi_{w,b}=u_{e,j_e}.
\]
Then
\begin{equation}
 Z_G((\tau_v),(T_e))
 :=\sum_{(j_e)}\left(\prod_e s_{e,j_e}\right)
       \prod_v\langle\tau_v,\xi_{v,1}\otimes\cdots\otimes\xi_{v,d_v}\rangle.
 \label{eq:weighted-contraction-definition}
\end{equation}
The sum of the absolute values is at most
\[
 \prod_v\|\tau_v\|_2\sum_{(j_e)}\prod_e s_{e,j_e}
 =\prod_v\|\tau_v\|_2\prod_e\|T_e\|_1.
\]
For finite-rank inputs this is the usual contraction with
$\langle T_ex,y\rangle$. Approximation in Hilbert and trace norms
proves independence of the decompositions and the continuity assertion.
\end{proof}

A threshold $\delta>0$ is regular if $\delta\notin\sigma(D)$.
For such a threshold, put
\[
 P_\delta=\mathbf1_{[\delta,\infty)}(D),\qquad
 P'_\delta=\mathbf1_{[\delta,\infty)}(D').
\]
These projections have the same finite rank.

\begin{lemma}[Contractions of the truncated tensors]
\label{lem:spectral-cap-opening}
The loop-free contractions determine every complete contraction,
including self-pairings, of
\begin{equation}
 \left(P_\delta^{\odot q}f_q\ (1\le q\le m),\quad
 B_\delta=P_\delta DP_\delta\right).
 \label{eq:finite-compressed-packet}
\end{equation}
Corresponding contractions of the primed family have the same values.
\end{lemma}

\begin{proof}
Choose a continuous $h$ on $[0,\Lambda]$, where
$\Lambda=\max(\|D\|,\|D'\|)$, such that $h=0$ near zero and
$h(x)=\mathbf1_{[\delta,\infty)}(x)/x$ on the common positive spectrum.
For polynomials $p_n\to h$ uniformly, set $Q_n=Dp_n(D)$. Then
\[
 P_\delta=Dh(D),\qquad
 \|Q_n-P_\delta\|_1
 \le\Tr D\,\|p_n-h\|_\infty\longrightarrow0.
\]
A contraction of \eqref{eq:finite-compressed-packet} is a contraction
of the uncompressed $f_q,D$ with $P_\delta$ on each edge.
Replacing $P_\delta$ by $Q_n$ converges by
\cref{lem:trace-weighted-network-bound}.

Since $Q_n=\sum_{\ell\ge1}c_{n,\ell}D^\ell$,
\[
 (Q_n)_{ij}=\sum_{\ell\ge1}c_{n,\ell}
 \sum_{i_1,\ldots,i_{\ell-1}}
 D_{ii_1}D_{i_1i_2}\cdots D_{i_{\ell-1}j}.
\]
Substitute \eqref{eq:support-operator-coefficients} in each factor.
For a self-pairing, the first replacement is explicitly
\[
 \sum_{i,j}g_{ij\alpha}D_{ij}
 =\sum_{q,i,j,\beta}g_{ij\alpha}(f_q)_{i\beta}(f_q)_{j\beta}.
\]
Every summed index now belongs to two different tensor copies.
The same holds for every positive power of $D$, including when $g=D$
is also expanded. Thus each approximating contraction is a finite
linear combination of the known loop-free contractions. The same
polynomials work for $D'$, so their limits agree.
\end{proof}

\subsection{Finite-dimensional invariant theory and compatible isomorphisms}
\label{subsec:invariant-theory}

\begin{proposition}[Separation of orthogonal tensor orbits]
\label{prop:finite-orthogonal-orbit-separation}
Let $V$ be finite-dimensional and
$\mathcal W=\bigoplus_{\alpha=1}^sV^{\otimes d_\alpha}$, with
$d_\alpha\ge0$. Two families in $\mathcal W$ belong to the same
$O(V)$-orbit if and only if all complete pairwise contractions, with
arbitrary repetitions, agree.
\end{proposition}

\begin{proof}
The orthogonal first fundamental theorem gives
\[
 \mathbb R[\mathcal W]^{O(V)}
 =\Span\{\text{complete contraction polynomials}\}
\]
\cite{WeylClassicalGroups,Brauer1937,Schrijver2008}.
If two compact orbits $\mathcal O_1,\mathcal O_2$ are disjoint,
polynomial approximation on their union gives $p$ with
$|p|<1/3$ on $\mathcal O_1$ and $|p-1|<1/3$ on $\mathcal O_2$.
Haar averaging,
\[
 \overline p(x)=\int_{O(V)}p(U\cdot x)\,dU,
 \qquad U\cdot(x_\alpha)_\alpha=(U^{\otimes d_\alpha}x_\alpha)_\alpha,
\]
gives an invariant polynomial with the same inequalities. Equal
contractions therefore force the orbits to coincide.
\end{proof}

The preceding two results give, at each regular threshold, an isometry
$U_\delta:P_\delta H\to P'_\delta H'$ satisfying
\begin{align}
 U_\delta^{\odot q}P_\delta^{\odot q}f_q
 &=(P'_\delta)^{\odot q}f'_q,
 &&1\le q\le m,
 \label{eq:finite-tensor-transport}\\
 U_\delta B_\delta U_\delta^*&=B'_\delta.
 \label{eq:finite-marker-transport}
\end{align}
The second identity gives
$U_\delta\ker(B_\delta-\lambda I)=\ker(B'_\delta-\lambda I)$.
Let
\[
 \mathcal K=\{U:H\to H'\text{ orthogonal}:UD=D'U\}.
\]
After one spectral identification of $H$ and $H'$, this space is
$\prod_{\lambda>0}O(\ker(D-\lambda I))$. Each factor is
finite-dimensional, so $\mathcal K$ is compact in the strong operator
topology. Every $U_\delta$ extends to an element of $\mathcal K$ by
choosing orthogonal maps on the remaining eigenspaces.

For regular $\delta_n\downarrow0$, set
\[
 \mathcal K_n=\{U\in\mathcal K:
 U^{\odot q}P_{\delta_n}^{\odot q}f_q
 =(P'_{\delta_n})^{\odot q}f'_q\text{ for }1\le q\le m\}.
\]
Each $\mathcal K_n$ is nonempty and closed, and
$\mathcal K_{n+1}\subset\mathcal K_n$. Compactness gives
$U\in\bigcap_n\mathcal K_n$. For every $q$,
\[
 \|U^{\odot q}f_q-f'_q\|_2
 \le\|f_q-P_{\delta_n}^{\odot q}f_q\|_2
    +\|f'_q-(P'_{\delta_n})^{\odot q}f'_q\|_2
 \longrightarrow0.
\]
Hence
\begin{equation}
 U^{\odot q}f_q=f'_q,\qquad 1\le q\le m.
 \label{eq:global-tensor-transport}
\end{equation}
This proves (ii)$\Rightarrow$(iii).

For (iii)$\Rightarrow$(i), transport a correlated Gaussian family by
$W'_a(h')=W_a(U^*h')$. Then
\[
 I_q^{W'_a}(f'_q)=I_q^{W_a}(f_q),
\]
so all joint laws agree. Finally, two transporting maps differ by an
orthogonal map fixing every tensor. This proves
\cref{thm:scalar-replica-tannaka}.\qed

\begin{remark}[Relation with closed tensor networks]
For the finite-dimensional contraction criterion, see also the real
Holant theorem \cite{YoungOrthogonalHolant}. Here the observations
initially give only loop-free contractions. The positive operator $D$
recovers self-pairings on its spectral subspaces and reduces the
infinite-dimensional problem to that criterion.
\end{remark}

\section{Reconstruction with a fixed Gaussian covariance kernel}
\label{subsec:universal-gaussian-canvas}

It remains to obtain the same contractions from one scalar process.
The correlations $\Sigma_{ab}=e^{-(u_a-u_b)^2}$ suffice.

\subsection{Construction and continuity}

Put
\begin{equation}
 K(u,v)=e^{-(u-v)^2},\qquad u,v\in\R.
 \label{eq:universal-gaussian-kernel}
\end{equation}
Define
\begin{equation}
 \Psi(u)=\left(e^{-u^2}\frac{2^{j/2}u^j}{\sqrt{j!}}\right)_{j\ge0}
 \in\ell^2(\mathbb N_0).
 \label{eq:gaussian-kernel-fock-factorization}
\end{equation}
Then
\[
 \langle\Psi(u),\Psi(v)\rangle
 =e^{-u^2-v^2}\sum_{j\ge0}\frac{(2uv)^j}{j!}
 =e^{-(u-v)^2}=K(u,v)
\]
\cite{Schoenberg1938}.
For distinct $u_1,\ldots,u_r$,
$\det(u_a^{j-1})_{j,a=1}^r=\prod_{a<b}(u_b-u_a)\ne0$.
Thus $\Psi(u_1),\ldots,\Psi(u_r)$ are linearly independent, and their
Gram matrix $(K(u_a,u_b))_{a,b}$ is positive definite.
A dilation characterization appears in
\cref{prop:canonical-gaussian-canvas}.

Let $\mathscr G_K=\overline{\Span}\{\Psi(u):u\in\R\}$ and let
$\mathbb X$ be isonormal over $\mathscr G_K\otimes H$. Setting
$\mathbb W_u(h)=\mathbb X(\Psi(u)\otimes h)$ gives
\begin{equation}
 \E[\mathbb W_u(h)\mathbb W_v(k)]=K(u,v)\ip hk_H.
 \label{eq:canvas-isonormal-covariance}
\end{equation}
Define
\begin{equation}
 \mathcal Z_{\mathbf f}(u)
 =f_0+\sum_{q=1}^mI_q^{\mathbb W_u}(f_q).
 \label{eq:universal-canvas-field}
\end{equation}
The space $C(\R)$ carries the topology of locally uniform convergence.

\begin{lemma}[Smooth sample paths and continuity in the tensors]
\label{lem:continuous-universal-canvas}
For degree-$m$ families $\mathbf f,\mathbf g$ on the same Hilbert
space, realized on the same Gaussian family, the processes admit smooth
modifications. For $p>2$ and $R>0$,
\begin{equation}
 \E\|\mathcal Z_{\mathbf f}-\mathcal Z_{\mathbf g}\|_{C([-R,R])}^p
 \le C_{m,p,R}
 \left(|f_0-g_0|^2+\sum_{q=1}^mq!\|f_q-g_q\|_2^2\right)^{p/2}.
 \label{eq:canvas-quantitative-continuity}
\end{equation}
Their laws on $C(\R)$ do not depend on the continuous modifications.
\end{lemma}

\begin{proof}
Put $T_uh=\Psi(u)\otimes h$ and $h_q=f_q-g_q$. Then
\begin{equation}
 I_q^{\mathbb W_u}(h_q)=I_q^{\mathbb X}(T_u^{\odot q}h_q).
 \label{eq:canvas-single-isonormal-realization}
\end{equation}
The series \eqref{eq:gaussian-kernel-fock-factorization} and all its
derivatives converge locally in Hilbert norm. Thus
$F_q(u)=T_u^{\odot q}h_q$ satisfies
$\sup_{|u|\le R}\|F_q^{(j)}(u)\|\le C_{q,j,R}\|h_q\|$.
Writing
\[
 S^2=|f_0-g_0|^2+\sum_qq!\|h_q\|^2,\qquad
 D_j(u)=\mathbf1_{\{j=0\}}(f_0-g_0)
       +\sum_qI_q^{\mathbb X}(F_q^{(j)}(u)),
\]
Wiener isometry and hypercontractivity give
\[
 \sup_{|u|\le R}\|D_j(u)\|_p\le C_{m,p,j,R}S,\qquad
 D_j(u)-D_j(a)=\int_a^uD_{j+1}(w)\,dw
 \quad\text{in }L^p(\Omega).
\]
Fubini identifies the $D_j$ as samplewise weak derivatives. Hence,
for every integer $k\ge0$,
\[
 \E\|D_0\|_{W^{k,p}([-R,R])}^p\le C_{m,p,k,R}S^p.
\]
Sobolev embedding gives $C^k$ representatives for every $k$.
Local averages converge both to these representatives and to $D_0(u)$
in $L^p$ at each fixed $u$, so they are modifications. Countably many
$k,R$ yield a smooth modification; $W^{1,p}\hookrightarrow C$ yields
\eqref{eq:canvas-quantitative-continuity}. Applying this also to each
process, rational evaluations identify their difference and determine
their laws on $C(\R)$.
\end{proof}

We henceforth use these smooth modifications.

\subsection{Recovery of contraction coefficients}

For $r\ge2$, let $E_r=\{u\in\R^r:\sum_au_a=0\}$.
For a symmetric zero-diagonal array $N$, put
\[
 \begin{gathered}
 (L_N)_{aa}=\sum_{b\ne a}N_{ab},\qquad
 (L_N)_{ab}=-N_{ab}\quad(a\ne b),\\
 Q_N(u)=u^{\mathsf T}L_Nu=\sum_{a<b}N_{ab}(u_a-u_b)^2,
 \qquad \chi_N=e^{-Q_N}.
 \end{gathered}
\]
In particular, $\chi_{N+N'}=\chi_N\chi_{N'}$.

\begin{lemma}[Linear independence of the functions $e^{-Q_N}$]
\label{lem:Gaussian-graph-character-faithfulness}
For every finite set of distinct nonnegative integral arrays $N$, the
functions $\chi_N$ are linearly independent on $E_r$. Equivalently,
\[
 \R[T_{ab}:a<b]\longrightarrow C^\omega(E_r),\qquad
 T_{ab}\longmapsto e^{-(u_a-u_b)^2}
\]
is injective.
\end{lemma}

\begin{proof}
If $Q_N-Q_{N'}$ vanishes on $E_r$, polarization gives
$(L_N-L_{N'})|_{E_r}=0$; also $(L_N-L_{N'})\mathbf1=0$.
Thus $L_N=L_{N'}$, and its off-diagonal entries give $N=N'$.
For finitely many distinct arrays, choose $w\in E_r$ outside the zero
sets of their pairwise differences. Then $\lambda_N=Q_N(w)$ are
distinct. Restriction to $u=\theta w$ turns a linear relation into
\[
 \sum_Nc_Ne^{-\lambda_Ns}=0,\qquad s=\theta^2\ge0.
\]
The derivatives of orders $0,\ldots,\#\{N\}-1$ at zero form a
Vandermonde system, so every $c_N$ vanishes.
\end{proof}

Set $E_1=\{0\}$. Let $\mathcal A_{r,m}$ be the finite set of
nonnegative integral symmetric zero-diagonal arrays with row sums at
most $m$. Equip $E_r$ with
\[
 d\gamma_r^\circ(u)=\pi^{-(r-1)/2}e^{-|u|^2}\,du,
\]
using induced Euclidean volume, and write $\det_{E_r}$ for the
determinant on $E_r$.

\begin{proposition}[Recovery of contraction coefficients]
\label{prop:canvas-recovers-loop-free-register}
For $\mathcal M_{\mathbf f,r}(u)=\E\prod_a\mathcal Z_{\mathbf f}(u_a)$,
\begin{equation}
 \mathcal M_{\mathbf f,r}(u)
 =\sum_{N\in\mathcal A_{r,m}}C_{\mathbf f}(N)\chi_N(u),
 \label{eq:canvas-character-expansion}
\end{equation}
where, with $q_a=\sum_{b\ne a}N_{ab}$,
\begin{equation}
 C_{\mathbf f}(N)=\frac{\prod_aq_a!}{\prod_{a<b}N_{ab}!}
 \operatorname{Contr}_N(f_{q_1},\ldots,f_{q_r}).
 \label{eq:canvas-network-coefficient}
\end{equation}
An isolated vertex contributes $f_0$. The matrix and vector
\begin{align}
 G_{N,N'}
 &=\int_{E_r}\chi_N\chi_{N'}\,d\gamma_r^\circ
 =\det_{E_r}(I+L_N+L_{N'})^{-1/2},
 \label{eq:Gaussian-character-Gram}\\
 b_{\mathbf f,r}(N')
 &=\int_{E_r}\mathcal M_{\mathbf f,r}\chi_{N'}\,d\gamma_r^\circ
 \label{eq:Gaussian-character-observed-vector}
\end{align}
satisfy $G>0$ and
\begin{equation}
 (C_{\mathbf f}(N))_{N\in\mathcal A_{r,m}}
 =G^{-1}(b_{\mathbf f,r}(N))_{N\in\mathcal A_{r,m}}.
 \label{eq:canonical-Gaussian-network-decoder}
\end{equation}
Thus the process law determines $f_0$ and all loop-free contractions.
\end{proposition}

\begin{proof}
Substitute $\Sigma_{ab}=K(u_a,u_b)$ in Wick's formula and
\eqref{eq:exact-scalar-coefficient}, deleting isolated vertices first.
This gives \eqref{eq:canvas-character-expansion}--\eqref{eq:canvas-network-coefficient}.
The moments exist by hypercontractivity, and
\[
 |\mathcal M_{\mathbf f,r}(u)|\le\|\mathcal Z_{\mathbf f}(0)\|_r^r
\]
by H\"older's inequality. All displayed integrals are therefore finite.
Gaussian integration gives \eqref{eq:Gaussian-character-Gram};
\cref{lem:Gaussian-graph-character-faithfulness} gives $G>0$.
Integration of the moment expansion gives $b=GC$. Finally,
$f_0=\E\mathcal Z_{\mathbf f}(0)$.
\end{proof}

The decoder commutes with coordinate permutations. Any probability
measure of full support on $E_r$ gives an invertible Gram matrix;
$\gamma_r^\circ$ also gives the determinant formula and permutation
invariance. No bound on $\|G^{-1}\|$ is asserted.

\subsection{Reconstruction and finite observations}

\begin{theorem}[Fixed-kernel reconstruction]
\label{thm:universal-canvas-replica-tannaka}
Let $\mathbf f$ on $H$ and $\mathbf f'$ on $H'$ be finite symmetric
tensor families, extended by zero tensors to the same maximal degree
$m_*$. For the processes defined by \eqref{eq:universal-canvas-field},
\[
 \Law_{C(\R)}(\mathcal Z_{\mathbf f})
 =\Law_{C(\R)}(\mathcal Z_{\mathbf f'})
\]
if and only if $f_0=f'_0$ and there is an orthogonal isomorphism
\begin{equation}
 U:S(\mathbf f)\to S(\mathbf f'),\qquad
 U^{\odot q}f_q=f'_q\quad(1\le q\le m_*).
 \label{eq:canvas-orbit-equivalence}
\end{equation}
\end{theorem}

\begin{proof}
Equal process laws give equal loop-free contractions by
\cref{prop:canvas-recovers-loop-free-register};
\cref{thm:scalar-replica-tannaka} gives the isomorphism.
Conversely, set $\mathbb W'_u(h')=\mathbb W_u(U^*h')$ on the minimal
supports. The Gaussian covariance is preserved and
$I_q^{\mathbb W'_u}(f'_q)=I_q^{\mathbb W_u}(f_q)$.
The finite-dimensional distributions, hence the laws on $C(\R)$, agree.
\end{proof}

The kernel and the construction $H\mapsto\mathscr G_K\otimes H$ are
independent of the tensors and their degree. An isometry $V:H\to H'$
induces $\Id_{\mathscr G_K}\otimes V$ on the Gaussian Hilbert spaces.

\begin{remark}[Orthogonal sums and convolution]
\label{prop:canvas-scalarization-naturality}
For families $\mathbf f$ on $H$ and $\mathbf g$ on $G$, let
$\mathbf f\boxplus\mathbf g$ add their constants and embed their
positive-order tensors into $H\oplus G$. The orthogonal Gaussian
families are independent, and
\[
 \mathcal Z_{\mathbf f\boxplus\mathbf g}
 =\mathcal Z_{\mathbf f}+\mathcal Z_{\mathbf g}.
\]
Thus orthogonal sums correspond to convolution of laws on $C(\R)$.
\end{remark}

\begin{remark}[Number of correlation parameters]
\label{rem:canvas-data-reduction}
Unrestricted correlations have $r(r-1)/2$ parameters. The matrices
$\Sigma_{ab}=e^{-(u_a-u_b)^2}$ have only $r-1$ parameters modulo
translation, yet restriction to them preserves every polynomial
coefficient by \cref{lem:Gaussian-graph-character-faithfulness}.
A single equicorrelation parameter does not separate all monomials.
\end{remark}

\begin{corollary}[Separation by finitely many evaluations]
\label{cor:finite-canvas-witness}
If $f_0\ne f'_0$, or if no common orthogonal isomorphism transports
the tensor families on their minimal supports, then for some $r\ge1$
and Lebesgue-almost every $u\in E_r$ the vectors
\begin{equation}
 \bigl(\mathcal Z_{\mathbf f}(u_1),\ldots,\mathcal Z_{\mathbf f}(u_r)\bigr),
 \qquad
 \bigl(\mathcal Z_{\mathbf f'}(u_1),\ldots,\mathcal Z_{\mathbf f'}(u_r)\bigr)
 \label{eq:finite-canvas-witness-vectors}
\end{equation}
have different laws. The separating configurations form an open dense
set containing a rational configuration.
\end{corollary}

\begin{proof}
Different means give $r=1$. Otherwise,
\cref{thm:scalar-replica-tannaka,lem:Gaussian-graph-character-faithfulness}
give some $r$ for which
$\mathcal M_{\mathbf f,r}-\mathcal M_{\mathbf f',r}$ is a nonzero
real-analytic function on $E_r$. Its nonzero set is open, dense and of
full measure, and contains a rational point. On this set the joint laws
differ. The full separating set is open as well: it is the union over
$z\in\R^r$ of the sets where the continuous joint characteristic
functions differ.
\end{proof}

\section{Heat tensors and geometric reconstruction}
\label{sec:gauge-field}

\subsection{The heat tensor pair}
\label{subsec:heat-tensor-pair}

Let
\begin{equation}
 \mathfrak X=(M,g,E,\nabla^E,V_E),\qquad
 \mathcal L_{\mathfrak X}=(\nabla^E)^*\nabla^E+V_E,
 \label{eq:network-Laplace-type-datum}
\end{equation}
where $M$ is closed and connected, $E$ is a positive-rank Euclidean
bundle, $\nabla^E$ is metric, and $V_E$ is smooth and self-adjoint.
For $t>0$, put
\begin{equation}
 H_{\mathfrak X}=L^2(M,E;dv_g),\qquad
 A_{\mathfrak X,t}=e^{-t\mathcal L_{\mathfrak X}/2},\qquad
 R_{\mathfrak X,t,x}a=(A_{\mathfrak X,t}a)(x).
 \label{eq:network-heat-evaluation-map}
\end{equation}
We use the self-adjoint $L^2$ realization
$\operatorname{Dom}(\mathcal L_{\mathfrak X})=H^2(M,E)$, which is bounded
below. If $\mathcal L_{\mathfrak X}e_j=\lambda_je_j$ in an
orthonormal eigenbasis, then
\[
 A_{\mathfrak X,t}a
 =\sum_j e^{-t\lambda_j/2}\langle a,e_j\rangle e_j.
\]
Every factor $e^{-t\lambda_j/2}$ is positive; hence $A_{\mathfrak X,t}$
is injective and its range contains every finite spectral sum.
Heat regularity makes $R_{\mathfrak X,t,x}$ bounded, and its adjoint
is defined by
\[
 \langle R_{\mathfrak X,t,x}^*v,a\rangle_{H_{\mathfrak X}}
 =\langle v,(A_{\mathfrak X,t}a)(x)\rangle_{E_x}.
\]
For an orthonormal basis $(\varepsilon_\alpha)$ of $E_x$, set
\[
 \eta_{\mathfrak X,t,x}
 =\sum_\alpha(R_{\mathfrak X,t,x}^*\varepsilon_\alpha)^{\otimes2}.
\]
For $a,b\in H_{\mathfrak X}$,
\[
 \langle\eta_{\mathfrak X,t,x},a\otimes b\rangle
 =\langle (A_{\mathfrak X,t}a)(x),(A_{\mathfrak X,t}b)(x)\rangle_{E_x}.
\]
Thus $\eta_{\mathfrak X,t,x}$ is independent of the basis.
The map $x\mapsto\eta_{\mathfrak X,t,x}$ is continuous in Hilbert
tensor norm. Since $M$ is compact, the Bochner integrals
\begin{equation}
 Q_{\mathfrak X,t}=\int_M\eta_{\mathfrak X,t,x}\,dv_g(x),\qquad
 C_{\mathfrak X,t}=\int_M\eta_{\mathfrak X,t,x}\odot
                              \eta_{\mathfrak X,t,x}\,dv_g(x)
 \label{eq:gauge-integrated-tensors}
\end{equation}
exist in $H_{\mathfrak X}^{\odot2}$ and $H_{\mathfrak X}^{\odot4}$.

\begin{definition}[The heat tensor pair]
\label{def:master-heat-packet}
Set $\mathcal T_t(\mathfrak X)
=(H_{\mathfrak X},Q_{\mathfrak X,t},C_{\mathfrak X,t})$.
\end{definition}

The defining identities are
\begin{align}
 \langle Q_{\mathfrak X,t},a\otimes b\rangle
 &=\int_M\langle A_{\mathfrak X,t}a,A_{\mathfrak X,t}b\rangle_E\,dv_g\notag\\
 &=\langle e^{-t\mathcal L_{\mathfrak X}}a,b\rangle,
 \label{eq:gauge-quadratic-flattening}\\
 \langle C_{\mathfrak X,t},a^{\otimes4}\rangle
 &=\int_M|A_{\mathfrak X,t}a|_E^4\,dv_g
 =\|A_{\mathfrak X,t}a\|_{L^4(M,E)}^4.
 \label{eq:gauge-quartic-L4-identity}
\end{align}
Degree four is the smallest radial degree above two, since
$\mathbb R[E_x]^{O(E_x)}=\mathbb R[|v|^2]$.

For an isonormal process $W$, the Gaussian section
$X_{\mathfrak X,t}=A_{\mathfrak X,t}W$ is smooth almost surely:
for $c> -\inf\sigma(\mathcal L_{\mathfrak X})$ and every integer $k\ge0$,
\[
 \E\|(c+\mathcal L_{\mathfrak X})^{k/2}X_{\mathfrak X,t}\|_2^2
 =\Tr\big((c+\mathcal L_{\mathfrak X})^k
                      e^{-t\mathcal L_{\mathfrak X}}\big)<\infty.
\]
Define the Wick powers by
\[
 :|X_{\mathfrak X,t}(x)|^2:=I_2(\eta_{\mathfrak X,t,x}),\qquad
 :|X_{\mathfrak X,t}(x)|^4:
 :=I_4(\eta_{\mathfrak X,t,x}\odot\eta_{\mathfrak X,t,x}).
\]
These expressions are frame-independent, including for nonscalar
pointwise covariance. Hence
\[
 \mathcal G_{\mathfrak X,t}
 :=\int_M\big(:|X_{\mathfrak X,t}|^2:+:|X_{\mathfrak X,t}|^4:\big)\,dv_g
 =I_2(Q_{\mathfrak X,t})+I_4(C_{\mathfrak X,t}).
\]
Write
\[
 \Rep_t(\mathfrak X)
 :=\Rep(0,0,Q_{\mathfrak X,t},0,C_{\mathfrak X,t}),\qquad
 \mathcal Z_{\mathfrak X,t}(u)
 :=I_2^{\mathbb W_u}(Q_{\mathfrak X,t})
   +I_4^{\mathbb W_u}(C_{\mathfrak X,t}),
\]
using the fixed kernel $K(u,v)=e^{-(u-v)^2}$.

\subsection{Removing the heat regularization}

Put $Q=Q_{\mathfrak X,t}^\flat$, $C=C_{\mathfrak X,t}$ and
$\mathcal R=\Ran Q^{1/2}$. In an orthonormal eigenbasis
$Qe_j=\mu_je_j$, one has $\mu_j=e^{-t\lambda_j}>0$. Define
\[
 (-\log Q)v=\sum_j(-\log\mu_j)\langle v,e_j\rangle e_j,
 \qquad
 Q^{-1/2}v=\sum_j\mu_j^{-1/2}\langle v,e_j\rangle e_j,
\]
on the vectors for which the respective coefficient sequences are
square summable. The second domain is exactly $\mathcal R$, and
$Q^{1/2}=A_{\mathfrak X,t}$. Therefore
\begin{equation}
 B:=-\log Q=t\mathcal L_{\mathfrak X},\qquad
 \mathcal N(v)^4:=\langle C,(Q^{-1/2}v)^{\otimes4}\rangle
 =\|v\|_4^4,\quad v\in\mathcal R.
 \label{eq:heat-unsmoothing}
\end{equation}
Finite spectral sums lie in $\mathcal R$ and approximate every smooth
section in every Sobolev norm. They are therefore dense in $L^4$.

\begin{proposition}[Operator and norm reconstruction]
\label{prop:heat-unsmoothing}
For an orthogonal isomorphism $U:H_{\mathfrak X}\to H_{\mathfrak Y}$, the identities
\[
 U^{\odot2}Q_{\mathfrak X,t}=Q_{\mathfrak Y,s},\qquad
 U^{\odot4}C_{\mathfrak X,t}=C_{\mathfrak Y,s}
\]
hold if and only if
\[
 U\operatorname{Dom}(\mathcal L_{\mathfrak X})
 =\operatorname{Dom}(\mathcal L_{\mathfrak Y}),\qquad
 tU\mathcal L_{\mathfrak X}=s\mathcal L_{\mathfrak Y}U,
\]
and the restriction $U:L^4(M,E)\to L^4(N,F)$ is a surjective isometry.
\end{proposition}
\begin{proof}
Write $Q_X=Q_{\mathfrak X,t}^\flat$ and $Q_Y=Q_{\mathfrak Y,s}^\flat$.
The quadratic identity is $UQ_X=Q_YU$. Thus
\[
 Q_Xe=\mu e\quad\Longrightarrow\quad Q_Y(Ue)=\mu Ue,
\]
so the preceding eigenvector formulas give
\[
 U(-\log Q_X)=(-\log Q_Y)U,\qquad UQ_X^{1/2}=Q_Y^{1/2}U
\]
with equal domains under $U$. If $v=Q_X^{1/2}a$, then
\[
 \|Uv\|_4^4
 =\langle C_{\mathfrak Y,s},(Ua)^{\otimes4}\rangle
 =\langle C_{\mathfrak X,t},a^{\otimes4}\rangle
 =\|v\|_4^4.
\]
In particular,
\begin{equation}
 \|Uv\|_{L^4(N,F)}=\|v\|_{L^4(M,E)},\qquad v\in\mathcal R_X.
 \label{eq:gauge-L4-on-heat-range}
\end{equation}
Since $U\mathcal R_X=\mathcal R_Y$, completing its restriction to
$\mathcal R_X$ gives a surjective $L^4$-isometry. It agrees with $U$
on $L^4$, because $L^4\hookrightarrow L^2$.

Conversely, the generator identity gives $UQ_X^{1/2}=Q_Y^{1/2}U$ and
$UQ_X=Q_YU$ on the same eigenbasis. Preservation of $L^4$ gives
\[
 \langle C_{\mathfrak Y,s},(Ua)^{\otimes4}\rangle
 =\|Q_Y^{1/2}Ua\|_4^4
 =\|Q_X^{1/2}a\|_4^4
 =\langle C_{\mathfrak X,t},a^{\otimes4}\rangle.
\]
Substitute $a=\sum_{i=1}^4r_i a_i$ and compare the coefficients of
$r_1r_2r_3r_4$. This gives $U^{\odot4}C_{\mathfrak X,t}=C_{\mathfrak Y,s}$.
\end{proof}

\subsection{A vector-valued Lamperti lemma}

\begin{lemma}[Simultaneous $L^2$--$L^4$ rigidity]
\label{lem:vector-Lamperti}
Let $E\to M$ and $F\to N$ be positive-rank Euclidean bundles over
closed Riemannian manifolds. Suppose an orthogonal
$U_2:L^2(M,E)\to L^2(N,F)$ and a surjective linear isometry
$U_4:L^4(M,E)\to L^4(N,F)$ agree on a subspace dense for
$\|\cdot\|_2+\|\cdot\|_4$, whose image is also dense for that norm.
They then agree on $L^4$ and have the form
\[
 (Ua)(y)=J_y a(\phi(y)),
\]
where $\phi$ is a volume-preserving measurable isomorphism modulo null
sets and $J_y:E_{\phi(y)}\to F_y$ are measurable surjective isometries.
In particular, the ranks agree.
\end{lemma}
\begin{proof}
Agreement follows from $L^4\hookrightarrow L^2$. For sections $a,b$,
\[
 \|a+b\|_4^4+\|a-b\|_4^4-2\|a\|_4^4-2\|b\|_4^4
 =\int_M\big(4|a|^2|b|^2+8\langle a,b\rangle^2\big)\,dv_g.
\]
The right-hand side vanishes exactly when $|a||b|=0$ almost everywhere.
The left-hand side is unchanged by $U$; hence $U$ and $U^{-1}$ preserve
disjoint supports. For measurable
$A\subset M$, let $\mathcal B_A$ consist of the sections supported in
$A$. If $A$ is the essential union of supports of a family $\mathcal S$,
then $\mathcal B_A=\mathcal S^{\perp\perp}$, where $\perp$ denotes
disjointness. Define $\theta(A)$ by
$U\mathcal B_A=\mathcal B_{\theta(A)}$, modulo null sets. Then
\[
 \theta(M\setminus A)=N\setminus\theta(A),\qquad
 \theta\Big(\bigcup_j A_j\Big)=\bigcup_j\theta(A_j).
\]
The inverse operator gives the inverse of $\theta$. Since the bases
are standard Borel spaces, $\theta(A)=\phi^{-1}(A)$ for a measurable
isomorphism $\phi:N\to M$. Decomposing a section on $A$ and its
complement gives
$U(\mathbf1_Aa)=\mathbf1_{\phi^{-1}(A)}Ua$. Therefore
\begin{equation}
 U(ba)=(b\circ\phi)Ua,\qquad b\in L^\infty(M).
 \label{eq:Linfty-module-intertwining}
\end{equation}
The last equality follows by uniform approximation with simple
functions. On a trivializing set $A$ with orthonormal frame $e_1,\ldots,e_r$,
define $T_y$ by
\[
 T_y e_i(\phi(y))=U(\mathbf1_Ae_i)(y),\qquad y\in\phi^{-1}(A).
\]
Then \eqref{eq:Linfty-module-intertwining} gives
$(Ua)(y)=T_y a(\phi(y))$, first for bounded coefficients and then
for all $a\in L^4$ by approximation.

Let $\rho>0$ satisfy
$\int_{\phi^{-1}(A)}\rho\,dv_h=\int_A dv_g$.
For a constant frame vector $c$ and every measurable $A$ inside its
trivialization, preservation of the $L^p$ norm gives
\[
 \int_{\phi^{-1}(A)}|T_yc|^p\,dv_h
 =|c|^p\int_A dv_g
 =\int_{\phi^{-1}(A)}\rho(y)|c|^p\,dv_h,\qquad p=2,4.
\]
Hence
\[
 |T_yc|^2=\rho(y)|c|^2,\qquad
 |T_yc|^4=\rho(y)|c|^4
\]
almost everywhere. A countable dense set of $c$ makes the null set
independent of $c$. Hence $\rho^2=\rho$, so $\rho=1$ and $T_y$ is
isometric. The same construction for $U^{-1}$ gives its inverse on
each fiber. Set $J_y=T_y$.
\end{proof}

This is the simultaneous-norm form of vector-valued Lamperti rigidity
\cite{Sourour1978,GreimJamison1987}.

\subsection{Reconstruction of the geometric data and time}

\begin{theorem}[Unknown-time heat-tensor rigidity]
\label{thm:scalar-observable-hears-gauge-field}
Let $\mathfrak X=(M,g,E,\nabla^E,V_E)$ and
$\mathfrak Y=(N,h,F,\nabla^F,V_F)$ have positive-dimensional bases.
For arbitrary $t,s>0$, the following are equivalent:
\begin{enumerate}[label=\textup{(\roman*)}]
\item $\Law(\mathcal Z_{\mathfrak X,t})
       =\Law(\mathcal Z_{\mathfrak Y,s})$ on $C(\mathbb R)$;
\item an orthogonal isomorphism $U:H_{\mathfrak X}\to H_{\mathfrak Y}$ satisfies
\[
 U^{\odot2}Q_{\mathfrak X,t}=Q_{\mathfrak Y,s},\qquad
 U^{\odot4}C_{\mathfrak X,t}=C_{\mathfrak Y,s};
\]
\item $t=s$ and there are a Riemannian isometry $\phi:(N,h)\to(M,g)$
and a smooth orthogonal bundle isomorphism $J:\phi^*E\to F$ with
\begin{equation}
 \nabla^F=J\phi^*\nabla^EJ^{-1},\qquad
 V_F=J(\phi^*V_E)J^{-1}.
 \label{eq:network-geometric-equivalence}
\end{equation}
\end{enumerate}
Every prescribed $U$ in \textup{(ii)} has a unique such representation
\begin{equation}
 (Ua)(y)=J_y a(\phi(y)).
 \label{eq:network-unitary-is-geometric}
\end{equation}
These conditions are also equivalent to
$\Rep_t(\mathfrak X)=\Rep_s(\mathfrak Y)$.
If $t=s$ is prescribed, the equivalence holds in every dimension.
\end{theorem}
\begin{proof}
Since $\overline{\Ran Q_{\mathfrak X,t}^\flat}=H_{\mathfrak X}$,
the tensor pair has minimal support $H_{\mathfrak X}$, and likewise
for $\mathfrak Y$. Thus
\cref{thm:scalar-replica-tannaka,thm:universal-canvas-replica-tannaka}
give \textup{(i)}$\Leftrightarrow$\textup{(ii)} and the assertion about $\Rep$.
By \cref{prop:heat-unsmoothing,lem:vector-Lamperti}, an operator in
\textup{(ii)} has the form \eqref{eq:network-unitary-is-geometric}, where
$\phi$ preserves volume, $J_y$ is orthogonal, and
\begin{equation}
 tU\mathcal L_{\mathfrak X}=s\mathcal L_{\mathfrak Y}U.
 \label{eq:scaled-generator-intertwining}
\end{equation}

For every $k$, \eqref{eq:scaled-generator-intertwining} gives
$U\operatorname{Dom}(\mathcal L_{\mathfrak X}^k)
=\operatorname{Dom}(\mathcal L_{\mathfrak Y}^k)$.
By elliptic regularity,
\[
 C^\infty(M,E)=\bigcap_{k\ge1}\operatorname{Dom}(\mathcal L_{\mathfrak X}^k),
 \qquad UC^\infty(M,E)=C^\infty(N,F).
\]
For $b\in C^\infty(M)$, let $M_ba=ba$. Equation
\eqref{eq:Linfty-module-intertwining} reads
\begin{equation}
 UM_bU^{-1}=M_{b\circ\phi}.
 \label{eq:gauge-smooth-multiplier-conjugacy}
\end{equation}
Where a smooth section $v$ is nonzero,
\[
 b\circ\phi=\frac{\langle M_{b\circ\phi}v,v\rangle}{|v|^2}.
\]
The numerator is smooth, so $b\circ\phi$ has a smooth representative.
Let $\iota=(b_1,\ldots,b_k):M\hookrightarrow\mathbb R^k$ be a smooth
embedding. The smooth representatives of $b_j\circ\phi$ define a map
$\Phi:N\to\mathbb R^k$ with $\Phi=\iota\circ\phi$ almost everywhere.
Since $\iota(M)$ is closed, $\Phi(N)\subset\iota(M)$ by continuity.
Thus $\iota^{-1}\circ\Phi$ is a smooth representative of $\phi$.
Applying the same construction to $U^{-1}$ gives a smooth inverse;
the compositions equal the identity almost everywhere and therefore
everywhere. Thus $\phi$ is a diffeomorphism. For a local frame $e_i$
and a cutoff $\chi$, the smooth sections $U(\chi e_i)$ give the columns
of $J$ wherever $\chi\circ\phi=1$. Hence $J$ is smooth.

Put $n=\dim M=\dim N$. In coordinates, $\mathcal L_{\mathfrak X}$
has second-order part $-\sum_{i,j}g^{ij}\partial_i\partial_j$.
Its principal symbol is
\[
 \sigma_2(\mathcal L_{\mathfrak X})(x,\xi)
 =\sum_{i,j}g^{ij}(x)\xi_i\xi_j\Id_{E_x}
 =|\xi|_g^2\Id_{E_x}.
\]
Equality of these terms in \eqref{eq:scaled-generator-intertwining} gives
\[
 t|\xi|_g^2=s|(d\phi)^*\xi|_h^2,
 \qquad h=(s/t)\phi^*g.
\]
Since $\phi$ preserves volume,
\begin{equation}
 dv_h=\phi^*dv_g=(s/t)^{n/2}\phi^*dv_g.
 \label{eq:volume-calibrates-time}
\end{equation}
Thus $s=t$ when $n>0$. If the common time was prescribed, this step
is unnecessary. In either case, $\phi$ is an isometry and $U$ intertwines
the unscaled generators.

Set
\[
 \widetilde\nabla=J\phi^*\nabla^EJ^{-1},\qquad
 \widetilde V=J(\phi^*V_E)J^{-1},\qquad
 B_1=\nabla^F-\widetilde\nabla.
\]
In normal coordinates and a $\widetilde\nabla$-normal orthonormal frame,
the first-order part of
$(\nabla^F)^*\nabla^F-\widetilde\nabla^*\widetilde\nabla$ is
\[
 -2\sum_{j=1}^n B_1(X_j)\partial_{X_j}.
\]
Equality of the operators forces $B_1=0$, then $V_F=\widetilde V$
\cite[Section~2.1]{BerlineGetzlerVergne1992}.
For each $b$, the operator $UM_bU^{-1}=M_{b\circ\phi}$ determines
$b\circ\phi$. Since smooth functions separate points, this determines
$\phi$.
Then $J_ya(\phi(y))=(Ua)(y)$ determines $J_y$, since values of smooth
sections span each fiber.
Conversely, \textup{(iii)} gives, for $Ua=J(a\circ\phi)$,
\[
 \|Ua\|_p=\|a\|_p\quad(p=2,4),\qquad
 U\mathcal L_{\mathfrak X}=\mathcal L_{\mathfrak Y}U.
\]
Apply \cref{prop:heat-unsmoothing}.
\end{proof}

\begin{remark}[Calibration of the symbol by volume]
\label{rem:intrinsic-heat-time}
Let $B=-\log(Q_{\mathfrak X,t}^\flat)$. Define $g_B$ by
\[
 \sigma_2(B)(x,\xi)=|\xi|_{g_B}^2\Id_{E_x}.
\]
Since $B=t\mathcal L_{\mathfrak X}$, one has $g_B=t^{-1}g$.
The measure $\mu$ obtained in \cref{lem:vector-Lamperti} is $dv_g$.
Thus
\begin{equation}
 d\mu=t^{n/2}dv_{g_B},\qquad
 t=\left(\frac{\mu(M)}{\operatorname{Vol}_{g_B}(M)}\right)^{2/n},
 \quad n>0.
 \label{eq:intrinsic-time-formula}
\end{equation}
On a point, $V_E=0$ gives $Q^\flat=\Id$ and
$\langle C,a^{\otimes4}\rangle=|a|^4$ for every $t$.
For a fixed smooth positive measure $\mu$, let $L_{g,\mu}$ be the
operator with quadratic form $\int_M|df|_g^2\,d\mu$. Then
\[
 L_{c^2g,\mu}=c^{-2}L_{g,\mu},\qquad
 e^{-c^2tL_{c^2g,\mu}}=e^{-tL_{g,\mu}}.
\]
Both heat tensors are therefore unchanged under
$(g,t,\mu)\mapsto(c^2g,c^2t,\mu)$. This includes
$\mu=dv_g/\operatorname{vol}_g(M)$, which is unchanged by metric rescaling.
\end{remark}

\subsection{Scalar observables of even order}
\label{sec:gaussian-observable-hears-manifold}

Let $\Delta_g=-\operatorname{div}_g\nabla_g$ on
$H_M=L^2(M,dv_g;\mathbb R)$, and let $p_s^M$ be its heat kernel.
For $t>0$ and $r\ge1$, set
\[
 A_{M,t}=e^{-t\Delta_g/2},\qquad
 k_{M,t,x}=p_{t/2}^M(x,\mathord\cdot),\qquad
 f_{M,t}^{(r)}=\int_M k_{M,t,x}^{\otimes r}\,dv_g(x).
\]
Since $\|k_{M,t,x}\|_2^2=p_t^M(x,x)$ is bounded, the tensors belong
to $H_M^{\odot r}$. For $X_{M,t}(x)=W_M(k_{M,t,x})$,
\[
 \E[X_{M,t}(x)X_{M,t}(y)]
 =\int_Mp_{t/2}^M(x,z)p_{t/2}^M(y,z)\,dv_g(z)=p_t^M(x,y).
\]
The preceding Sobolev estimate gives a smooth modification.
Writing $\sigma_{M,t}(x)^2=p_t^M(x,x)$ and $\operatorname{He}_r$ for
the probabilists' Hermite polynomial,
\[
 :X_{M,t}(x)^r:
 =\sigma_{M,t}(x)^r\operatorname{He}_r
       \bigl(X_{M,t}(x)/\sigma_{M,t}(x)\bigr)
 =I_r(k_{M,t,x}^{\otimes r}).
\]
For even $q\ge4$, define
\begin{align*}
 F_{M,t,q}
 &=\int_M\bigl(:X_{M,t}(x)^2:+:X_{M,t}(x)^q:\bigr)\,dv_g(x)
 =I_2(f_{M,t}^{(2)})+I_q(f_{M,t}^{(q)}),\\
 \mathcal Z_{M,t,q}(u)
 &=I_2^{\mathbb W_{M,u}}(f_{M,t}^{(2)})
   +I_q^{\mathbb W_{M,u}}(f_{M,t}^{(q)}),\qquad u\in\mathbb R,
\end{align*}
with covariance kernel $e^{-(u-v)^2}$. Write
$\Rep_{t,q}(M,g)=\Rep(0,0,f_{M,t}^{(2)},0,\ldots,0,f_{M,t}^{(q)})$.
For $q=4$, these are the scalar heat tensors of the preceding section.

\Needspace{12\baselineskip}
\begin{corollary}[Scalar Laplacians and even-order observables]
\label{thm:scalar-observable-hears-manifold}
\label{cor:universal-canvas-hears-manifold}
Let $(M,g)$ and $(N,h)$ be closed connected positive-dimensional
Riemannian manifolds. For $t,s>0$ and even $q\ge4$, the following
are equivalent:
\begin{enumerate}[label=\textup{(\roman*)}]
\item $\mathcal Z_{M,t,q}$ and $\mathcal Z_{N,s,q}$ have the same law
on $C(\mathbb R)$;
\item $\Rep_{t,q}(M,g)=\Rep_{s,q}(N,h)$;
\item $t=s$ and $(M,g)$ is Riemannian isometric to $(N,h)$.
\end{enumerate}
For a prescribed common time $t=s$, the equivalence holds in every
dimension.
\end{corollary}
\begin{proof}
The identities
\[
 \bigl(f_{M,t}^{(2)}\bigr)^\flat=e^{-t\Delta_g},\qquad
 \langle f_{M,t}^{(q)},a^{\otimes q}\rangle
 =\|A_{M,t}a\|_q^q
\]
and injectivity of the heat operator give full minimal support.
By \cref{thm:scalar-replica-tannaka,thm:universal-canvas-replica-tannaka},
each of \textup{(i)} and \textup{(ii)} is equivalent to an orthogonal
isomorphism $U:H_M\to H_N$ with
\[
 U^{\odot2}f_{M,t}^{(2)}=f_{N,s}^{(2)},\qquad
 U^{\odot q}f_{M,t}^{(q)}=f_{N,s}^{(q)}.
\]
The first identity is $Ue^{-t\Delta_g}=e^{-s\Delta_h}U$.
On each common heat eigenspace, taking minus the logarithm of the
eigenvalue and its positive square root gives, respectively,
\[
 tU\Delta_g=s\Delta_hU,\qquad
 UA_{M,t}=A_{N,s}U,\qquad
 \|UA_{M,t}a\|_q=\|A_{M,t}a\|_q.
\]
The last equality follows by evaluating the order-$q$ tensor at
$a^{\otimes q}$. Finite spectral sums lie in both heat ranges and
approximate smooth functions in every Sobolev norm. Completion gives
a surjective $L^q$-isometry, agreeing with $U$ in $L^2$ because
$L^q\hookrightarrow L^2$. Banach--Lamperti \cite{Lamperti1958} gives
$Uf=a(f\circ\phi)$. Taking $f=\mathbf1_A$ in the two norm identities gives
\[
 \phi_*(|a|^2dv_h)=dv_g=\phi_*(|a|^qdv_h).
\]
Since $\phi$ is a measure-class isomorphism and $a\ne0$ almost
everywhere, $|a|=1$ and $\phi_*dv_h=dv_g$. Moreover,
$a=U\mathbf1\in\ker\Delta_h$, so connectedness makes $a$ a constant
sign. As in \cref{thm:scalar-observable-hears-gauge-field},
$UC^\infty(M)=C^\infty(N)$ and
$UM_bU^{-1}=M_{b\circ\phi}$ make $\phi$ a diffeomorphism. The
second-order coefficients and the volume identity give
\[
 h=(s/t)\phi^*g,\qquad
 dv_h=\phi^*dv_g=(s/t)^{\dim M/2}\phi^*dv_g.
\]
Hence $s=t$ in positive dimension and $\phi$ is an isometry. With a
prescribed common time, no dimension restriction is needed; see also
\cite{ArendtBiegertTerElst2012}.
Conversely, an isometry at a common time satisfies
$UA_{M,t}=A_{N,t}U$ and $\int_N(Uf)^r\,dv_h=\int_M f^r\,dv_g$ for
$r=2,q$, hence the two tensor identities.
\end{proof}

Multiplication and the recovered Laplacian determine the metric by
\[
 \Gamma_g(f_1,f_2)
 =\tfrac12\bigl(f_1\Delta_gf_2+f_2\Delta_gf_1
                   -\Delta_g(f_1f_2)\bigr)
 =\langle\nabla f_1,\nabla f_2\rangle_g.
\]
The corollary also holds when the two chaos components have fixed
known nonzero coefficients.

\begin{remark}[A cubic tensor in the scalar case]
\label{rem:native-scalar-cubic}
Set $T_{M,t}^{(3)}=\int_M k_{M,t,x}^{\otimes3}\,dv_g(x)$.
Suppose an orthogonal isomorphism $U:L^2(M)\to L^2(N)$ satisfies
\[
 U^{\odot2}f_{M,t}^{(2)}=f_{N,s}^{(2)},\qquad
 U^{\odot3}T_{M,t}^{(3)}=T_{N,s}^{(3)}.
\]
As above, $tU\Delta_g=s\Delta_hU$ and $UA_{M,t}=A_{N,s}U$.
Evaluating the cubic identity on three vectors gives
\[
 \int_N(Uf)(Ug)(Uh)\,dv_h=\int_Mfgh\,dv_g,
 \qquad f,g,h\in\Ran A_{M,t}.
\]
Since $\ker\Delta_g=\mathbb R\mathbf1=\ker\Delta_h$, write
$U\mathbf1=c\mathbf1$. The quadratic and cubic identities yield
\[
 \operatorname{vol}(M)=c^2\operatorname{vol}(N)
                     =c^3\operatorname{vol}(N),
\]
so $c=1$. For finite spectral sums $f,g,h$,
\[
 \langle U(fg),Uh\rangle
 =\langle fg,h\rangle
 =\int_N(Uf)(Ug)(Uh)\,dv_h.
\]
The vectors $Uh$ are dense in $L^2(N)$; hence
$U(fg)=(Uf)(Ug)$. Approximation by spectral sums in every Sobolev norm
gives this identity for all smooth $f,g$. For $y\in N$, set
$\chi_y(f)=(Uf)(y)$. If $\ker\chi_y$ had no common zero, compactness
would give $f_1,\ldots,f_k\in\ker\chi_y$ with $F=\sum_i f_i^2>0$.
Then $F^{-1}$ is smooth and
\[
 1=\chi_y(FF^{-1})=\chi_y(F)\chi_y(F^{-1})=0.
\]
Thus the kernel has a common zero $\phi(y)$. Since
$f-\chi_y(f)\mathbf1\in\ker\chi_y$, one has
$f(\phi(y))=\chi_y(f)=(Uf)(y)$.
Applying this identity to coordinates of an
embedding of $M$, and then to $U^{-1}$, proves that $\phi$ is a
diffeomorphism. Also
\[
 \int_Nf\circ\phi\,dv_h
 =\langle Uf,U\mathbf1\rangle=\int_Mf\,dv_g.
\]
The same coefficient comparison gives
$h=(s/t)\phi^*g$ and $(s/t)^{\dim M/2}=1$.
Thus the cubic pair determines the time in positive dimension and
the metric in every dimension at a prescribed common time.
The fourth-order construction uses $|v|^4$ and applies also to
finite-rank bundles and to \cref{sec:dirichlet-rcd}.
\end{remark}

\subsection{Scalar contractions of the heat tensors}
\label{sec:radial-heat-networks}

Let $\Gamma=(\mathsf V,\mathsf E)$ be a finite loop-free multigraph,
with parallel edges allowed and every vertex degree in $\{2,4\}$.
For one vector $a_{v,e}$ at each incident vertex-edge pair, define
\[
 \operatorname{Contr}_{\Gamma}
 \left(\left(\bigotimes_{e\ni v}a_{v,e}\right)_{v\in\mathsf V}\right)
 =\prod_{e=\{v,w\}}\langle a_{v,e},a_{w,e}\rangle.
\]
Extend by multilinearity and continuity to Hilbert tensors. Let
$\mathcal N_{\Gamma,t}(\mathfrak X)$ be this expression with
$Q_{\mathfrak X,t}$ at degree-two vertices and $C_{\mathfrak X,t}$
at degree-four vertices. For $\tau_{2,x}=\eta_{\mathfrak X,t,x}$ and
$\tau_{4,x}=\eta_{\mathfrak X,t,x}\odot\eta_{\mathfrak X,t,x}$,
\begin{equation}
 \mathcal N_{\Gamma,t}(\mathfrak X)
 =\int_{M^{\mathsf V}}
   \operatorname{Contr}_{\Gamma}
   \left((\tau_{d(v),x_v})_{v\in\mathsf V}\right)
   \prod_{v\in\mathsf V}dv_g(x_v).
 \label{eq:network-local-heat-kernel-formula}
\end{equation}
For the scalar Laplacian,
\[
 \mathcal N_{\Gamma,t}(M,g)
 =\int_{M^{\mathsf V}}
   \prod_{\{v,w\}\in\mathsf E}p_t^M(x_v,x_w)
   \prod_{v\in\mathsf V}dv_g(x_v),
\]
with edge multiplicities retained. For bundles,
\eqref{eq:network-local-heat-kernel-formula} contracts the bundle heat
kernels with the fiber metrics.

\begin{corollary}[Heat-network invariants]
\label{thm:radial-heat-network-rigidity}
For positive-dimensional heat geometries $\mathfrak X,\mathfrak Y$
and $t,s>0$,
\[
 \mathcal N_{\Gamma,t}(\mathfrak X)
 =\mathcal N_{\Gamma,s}(\mathfrak Y)
 \quad\text{for every connected }\Gamma
\]
if and only if $t=s$ and $\mathfrak X\simeq\mathfrak Y$.
For a prescribed common time, the statement holds in every dimension.
\end{corollary}
\begin{proof}
If $\Gamma=\Gamma_1\sqcup\cdots\sqcup\Gamma_k$, then
$\mathcal N_{\Gamma,t}=\prod_{j=1}^k\mathcal N_{\Gamma_j,t}$.
Thus all loop-free contractions agree. By
\cref{thm:scalar-replica-tannaka}, there is an orthogonal isomorphism $U$ with
$U^{\odot2}Q_{\mathfrak X,t}=Q_{\mathfrak Y,s}$ and
$U^{\odot4}C_{\mathfrak X,t}=C_{\mathfrak Y,s}$.
The supports are the full $L^2$ spaces, so
\cref{thm:scalar-observable-hears-gauge-field} applies.
Conversely, a geometric equivalence at a common time gives these tensor
identities. Every edge pairing is unchanged, since
$\langle Ua,Ub\rangle=\langle a,b\rangle$.
\end{proof}

\subsection{Automorphisms and geometric consequences}

Let $\operatorname{Aut}(\mathfrak X)$ consist of smooth orthogonal
bundle automorphisms $(\psi,\mathcal J)$ covering base isometries and
preserving $\nabla^E,V_E$. Their $L^2$ action is
\begin{equation}
 (\rho(\psi,\mathcal J)a)(y)
 =\mathcal J_{\psi^{-1}(y)}a(\psi^{-1}(y)).
 \label{eq:geometric-automorphism-L2-action}
\end{equation}
For $t>0$, put
\begin{equation}
 \Stab_t(\mathfrak X)
 =\{U\in O(H_{\mathfrak X}):
       U^{\odot2}Q_{\mathfrak X,t}=Q_{\mathfrak X,t},\quad
       U^{\odot4}C_{\mathfrak X,t}=C_{\mathfrak X,t}\}.
 \label{eq:radial-heat-tensor-stabilizer}
\end{equation}

\begin{corollary}[The tensor stabilizer]
\label{cor:radial-network-stabilizer}
The action $\rho$ is a group isomorphism
\[
 \operatorname{Aut}(\mathfrak X)\simeq\Stab_t(\mathfrak X).
\]
For the scalar Laplacian on a closed connected manifold,
\[
 \Stab_t(M,g)\cong\{\pm1\}\times\operatorname{Isom}(M,g).
\]
\end{corollary}
\begin{proof}
By \cref{thm:scalar-observable-hears-gauge-field}, each stabilizer
operator has a unique representation $Ua(y)=J_ya(\phi(y))$.
Set $\psi=\phi^{-1}$ and $\mathcal J_x=J_{\phi^{-1}(x)}$.
This proves bijectivity of the group action $\rho$.
In rank one, a smooth orthogonal fiber map is a constant sign;
these signs commute with isometric pullbacks.
\end{proof}

\begin{remark}[Moduli, holonomy, and finite separation]
\label{cor:gauge-moduli-holonomy-separation}
On a fixed closed connected base, the law determines a rank-$r$
metric connection up to base isometry and orthogonal gauge, also on
the Yang--Mills locus. It determines curvature and parallel transport:
\[
 P_{\nabla^F}(\gamma)
 =J_{\gamma(1)}P_{\nabla^E}(\phi\circ\gamma)J_{\gamma(0)}^{-1}
\]
for every piecewise smooth path $\gamma$ in $N$.
For a flat connection, $[\gamma]\mapsto P_{\nabla^E}(\gamma)$
is a representation of $\pi_1(M)$ in $O(r)$ after choosing an
orthonormal frame at the basepoint. Changing that frame conjugates
the representation. Hence
\[
 \frac{\operatorname{Hom}(\pi_1(M),O(r))/O(r)}
      {\operatorname{Isom}(M,g)}
 \hookrightarrow\{\text{scalar process laws at time }t\}.
\]
For $r=1$, the domain is $H^1(M;\mathbb Z_2)$ modulo isometries.
By \cref{cor:finite-canvas-witness}, two inequivalent heat geometries,
including isospectral ones, are separated by a joint law at finitely
many rational points.
\end{remark}

\begin{remark}[The orthogonal gauge group]
\label{rem:orthogonal-gauge-optimal}
Choose a closed connected $(M,g)$ of dimension at least two with
trivial isometry group, and a one-form $A$ with $dA\ne0$; a closed
surface of genus at least two admits such a metric. On the trivial
oriented rank-two bundle, put
\[
 J_0=\begin{pmatrix}0&-1\\1&0\end{pmatrix},\qquad
 \nabla^A=d+J_0A.
\]
A constant reflection $C\in O(2)\setminus SO(2)$ satisfies
$CJ_0C^{-1}=-J_0$, so it conjugates $\nabla^A$ to $\nabla^{-A}$ and
preserves every scalar process law. An $SO(2)$ gauge changes $A$ by
a closed one-form. Since $-2A$ is not closed, no such gauge identifies
the two connections; the only possible base isometry is the identity.
Thus a prescribed fiber orientation or complex structure requires
additional tensors.
\end{remark}

\section{Dirichlet forms and compact RCD spaces}
\label{sec:dirichlet-rcd}

\subsection{An abstract Dirichlet-space theorem}

For $i=1,2$, let $(X_i,\mu_i)$ be a finite standard Borel measure
space and $P_r^{(i)}=e^{-rL_i}$ a strongly continuous symmetric Markov
semigroup on $H_i=L^2(X_i,\mu_i)$, with $L_i\ge0$ self-adjoint.
Let $t_i>0$. Assume:
\begin{enumerate}[label=\textnormal{(D\arabic*)}]
\item $P_r^{(i)}\mathbf1=\mathbf1$ for $r\ge0$, and
\[
 \bigl[P_r^{(i)}\mathbf1_A=\mathbf1_A\text{ for all }r\ge0\bigr]
 \Longrightarrow \mu_i(A)\mu_i(X_i\setminus A)=0;
\]
\item $P_{t_i/2}^{(i)}$ has a symmetric measurable kernel:
\[
 (P_{t_i/2}^{(i)}f)(x)
 =\int_{X_i}p_{t_i/2}^{(i)}(x,y)f(y)\,d\mu_i(y)
 =\ip{k_{i,t_i,x}}{f}_{H_i},
\]
where $x\mapsto k_{i,t_i,x}:=p_{t_i/2}^{(i)}(x,\mathord\cdot)$ is
strongly measurable into $H_i$;
\item $\displaystyle\int_{X_i}\norm{k_{i,t_i,x}}_2^4\,d\mu_i(x)<\infty$.
\end{enumerate}
By (D3) and finiteness of $\mu_i$,
\[
 \int_{X_i}\norm{k_{i,t_i,x}}_2^2\,d\mu_i(x)
 \le\mu_i(X_i)^{1/2}
 \left(\int_{X_i}\norm{k_{i,t_i,x}}_2^4\,d\mu_i(x)\right)^{1/2}<\infty.
\]
Thus the Bochner integrals
\[
 Q_{i,t_i}=\int_{X_i}k_{i,t_i,x}^{\otimes2}\,d\mu_i(x),
 \qquad
 C_{i,t_i}=\int_{X_i}k_{i,t_i,x}^{\otimes4}\,d\mu_i(x)
\]
belong to $H_i^{\odot2}$ and $H_i^{\odot4}$. Moreover,
\[
 \norm{P_{t_i/2}^{(i)}f}_4^4
 \le \norm f_2^4\int_{X_i}\norm{k_{i,t_i,x}}_2^4\,d\mu_i(x).
\]
For $g\in L^{4/3}(X_i,\mu_i)$, symmetry and duality give
\[
 \left[\int_{X_i}gP_{t_i/2}^{(i)}f\,d\mu_i=0\quad(f\in H_i)\right]
 \Longrightarrow P_{t_i/2}^{(i)}g=0
 \Longrightarrow P_r^{(i)}g=0\quad(r\ge t_i/2).
\]
Analyticity on $L^{4/3}$ \cite{BakryGentilLedoux2014} gives
$P_r^{(i)}g=0$ for every $r>0$; strong continuity gives $g=0$.
Hence
\[
 \overline{P_{t_i/2}^{(i)}H_i}^{\,L^4}=L^4(X_i,\mu_i).
\]

Using the Gaussian covariance \eqref{eq:canvas-isonormal-covariance}, set
\begin{equation}
 \mathcal Z_{i,t_i}(u)
 =I_2^{\mathbb W_{i,u}}(Q_{i,t_i})
  +I_4^{\mathbb W_{i,u}}(C_{i,t_i}),\qquad u\in\mathbb R.
 \label{eq:dirichlet-canvas-observable}
\end{equation}

\begin{theorem}[Markov semigroup rigidity]
\label{thm:dirichlet-canvas-reconstruction}
Under \textnormal{(D1)}--\textnormal{(D3)}, the following are equivalent.
\begin{enumerate}[label=\textnormal{(\roman*)}]
\item $\mathcal Z_{1,t_1}$ and $\mathcal Z_{2,t_2}$ have the same law
on $C(\mathbb R)$.
\item There is a measurable isomorphism $\phi:X_2\to X_1$ modulo
null sets, with $\phi_*\mu_2=\mu_1$, whose pullback
\begin{equation}
 Vf=f\circ\phi
 \label{eq:dirichlet-spatial-unitary}
\end{equation}
satisfies
\[
 VP_{rt_1}^{(1)}=P_{rt_2}^{(2)}V\qquad(r\ge0).
\]
\item Such a pullback satisfies
\[
 V\operatorname{Dom}(\mathcal E_1)=\operatorname{Dom}(\mathcal E_2),
 \qquad t_2\mathcal E_2(Vf,Vg)=t_1\mathcal E_1(f,g).
\]
\end{enumerate}
Here $\operatorname{Dom}(\mathcal E_i)=\operatorname{Dom}(L_i^{1/2})$
and $\mathcal E_i(f,g)=\ip{L_i^{1/2}f}{L_i^{1/2}g}$.
For $t_1=t_2$, these are unscaled semigroup and form equivalences.
\end{theorem}

\begin{proof}
For $f,g\in H_i$,
\begin{align}
 \ip{Q_{i,t_i}^\flat f}{g}
 &=\int_{X_i}\ip{k_{i,t_i,x}}{f}\ip{k_{i,t_i,x}}{g}\,d\mu_i(x)
 =\ip{P_{t_i}^{(i)}f}{g},
 \label{eq:dirichlet-quadratic-is-heat}\\
 \ip{C_{i,t_i}}{f^{\otimes4}}
 &=\norm{P_{t_i/2}^{(i)}f}_4^4.
 \label{eq:dirichlet-quartic-is-L4}
\end{align}
Since $\ker Q_{i,t_i}^\flat=\ker e^{-t_iL_i}=\{0\}$, the minimal
support is $H_i$. Thus \textnormal{(i)} and
\cref{thm:universal-canvas-replica-tannaka} give an orthogonal
isomorphism $U:H_1\to H_2$ with
\begin{equation}
 U^{\odot2}Q_{1,t_1}=Q_{2,t_2},\qquad
 U^{\odot4}C_{1,t_1}=C_{2,t_2}.
 \label{eq:dirichlet-two-four-unitary}
\end{equation}
The quadratic identity gives $Ue^{-t_1L_1}=e^{-t_2L_2}U$.
Applying $-\log$ and then $\lambda\mapsto e^{-r\lambda}$ gives
\begin{equation}
 \begin{gathered}
 U\operatorname{Dom}(L_1)=\operatorname{Dom}(L_2),\qquad
 t_1UL_1=t_2L_2U,\\
 UP_{rt_1}^{(1)}=P_{rt_2}^{(2)}U\qquad(r\ge0).
 \end{gathered}
 \label{eq:dirichlet-unitary-all-times}
\end{equation}
Taking $r=1/2$ in this identity and using the quartic tensor gives
\begin{equation}
 \norm{UP_{t_1/2}^{(1)}f}_4
 =\norm{P_{t_1/2}^{(1)}f}_4.
 \label{eq:dirichlet-L4-on-heat-range}
\end{equation}
By density, the restriction of $U$ to the heat range extends to a
surjective $L^4$ isometry. It agrees with $U$ on $L^4$ because
$L^4\hookrightarrow L^2$.
Banach--Lamperti \cite{Lamperti1958} gives $Uf=a(f\circ\phi)$,
where $\phi$ is a measure-class isomorphism and $a\ne0$ almost
everywhere. Applying the two norm identities to $\mathbf1_A$ gives
\[
 \int_{\phi^{-1}(A)}|a|^2\,d\mu_2
 =\mu_1(A)
 =\int_{\phi^{-1}(A)}|a|^4\,d\mu_2
 \quad(A\subset X_1\text{ measurable}).
\]
Thus $|a|^2=|a|^4$, so $|a|=1$ and $\phi_*\mu_2=\mu_1$.
Moreover,
\[
 P_{rt_2}^{(2)}a
 =P_{rt_2}^{(2)}U\mathbf1
 =UP_{rt_1}^{(1)}\mathbf1
 =U\mathbf1=a\qquad(r\ge0).
\]
Since $\mathbf1_{\{a=1\}}=(1+a)/2$, (D1) gives
$a=\varepsilon\in\{-1,1\}$. Hence
$V=\varepsilon U$, $Vf=f\circ\phi$, proves \textnormal{(ii)}.

For a pullback in \textnormal{(ii)}, functional calculus gives
\[
 (t_2L_2)^{1/2}V=V(t_1L_1)^{1/2},\qquad
 t_2\mathcal E_2(Vf,Vg)=t_1\mathcal E_1(f,g),
\]
with equality of the square-root domains. Conversely, the unique
self-adjoint operators associated with the forms $t_i\mathcal E_i$
are $t_iL_i$ \cite{FukushimaOshimaTakeda2011}; thus
\textnormal{(iii)} implies \textnormal{(ii)}.
Finally, \textnormal{(ii)} gives
\begin{align*}
 \ip{Q_{2,t_2}^\flat Vf}{Vg}&=\ip{Q_{1,t_1}^\flat f}{g},\\
 \ip{C_{2,t_2}}{(Vf)^{\otimes4}}
 &=\norm{V P_{t_1/2}^{(1)}f}_4^4
 =\ip{C_{1,t_1}}{f^{\otimes4}}.
\end{align*}
Polarization gives $V^{\odot2}Q_{1,t_1}=Q_{2,t_2}$ and
$V^{\odot4}C_{1,t_1}=C_{2,t_2}$;
\cref{thm:universal-canvas-replica-tannaka} gives \textnormal{(i)}.
\end{proof}

\subsection{Compact \texorpdfstring{$\operatorname{RCD}$}{RCD} spaces}

\begin{proposition}[The hypotheses for compact RCD spaces]
\label{prop:RCD-reconstruction-hypotheses}
Let $(X,d,\mu)$ be a compact connected $\operatorname{RCD}(K,N)$ space
with full support and $N<\infty$. For every $t>0$, its Cheeger heat
semigroup satisfies \textnormal{(D1)}--\textnormal{(D3)}.
\end{proposition}

\begin{proof}
The Cheeger form is regular, strongly local and quadratic; its heat
semigroup is symmetric, conservative and Markovian
\cite{AmbrosioGigliSavare2014Invent,AmbrosioGigliSavare2014Duke,
ErbarKuwadaSturm2015}. Positivity of the heat kernel gives
irreducibility. Gaussian bounds and compactness give
\cite{JiangLiZhang2016}
\[
 \norm{p_{t/2}(x,\mathord\cdot)}_2^2=p_t(x,x),\qquad
 \sup_xp_t(x,x)<\infty.
\]
Consequently,
\[
 \int_X\norm{p_{t/2}(x,\mathord\cdot)}_2^4\,d\mu(x)
 =\int_Xp_t(x,x)^2\,d\mu(x)
 \le\mu(X)\sup_xp_t(x,x)^2<\infty,
\]
which is (D3). Joint measurability, separability of $L^2$ and
Pettis' theorem give (D2). Also,
\[
 \norm{P_{t/2}f}_\infty
 \le\sup_xp_t(x,x)^{1/2}\norm f_2.
\]
\end{proof}

\begin{corollary}[Compact RCD spaces]
\label{cor:canvas-hears-compact-RCD}
Let $(X_i,d_i,\mu_i)$ be compact connected
$\operatorname{RCD}(K_i,N_i)$ spaces with full support and
$N_i<\infty$, and let $t_i>0$. Then
\[
 \Law_{C(\mathbb R)}(\mathcal Z_{1,t_1})
 =\Law_{C(\mathbb R)}(\mathcal Z_{2,t_2})
\]
if and only if there is a measure-preserving isometry
\[
 \phi:(X_2,t_2^{-1/2}d_2,\mu_2)
       \longrightarrow(X_1,t_1^{-1/2}d_1,\mu_1).
\]
In particular, at a common time the law determines the metric measure
space. The parameters $K_i,N_i$ may differ; no assumption on the
regular sets is needed.
\end{corollary}

\begin{proof}
For the distance $t_i^{-1/2}d_i$, the minimal weak upper gradient,
Cheeger form and generator satisfy
\[
 \begin{gathered}
 |Df|_{t_i^{-1/2}d_i}=t_i^{1/2}|Df|_{d_i},\\
 \mathcal E_{t_i^{-1/2}d_i,\mu_i}=t_i\mathcal E_i,\qquad
 L_{t_i^{-1/2}d_i,\mu_i}=t_iL_i.
 \end{gathered}
\]
By \cref{prop:RCD-reconstruction-hypotheses,thm:dirichlet-canvas-reconstruction},
equality of laws gives $\phi_*\mu_2=\mu_1$ and
$t_2\mathcal E_2(f\circ\phi,g\circ\phi)=t_1\mathcal E_1(f,g)$.
The Sobolev-to-Lipschitz property gives \cite{Gigli2018}
\[
 t_i^{-1/2}d_i(x,y)
 =\sup\bigl\{|f(x)-f(y)|:
 f\in W^{1,2}(X_i)\cap C(X_i),\ t_i|Df|_{d_i}^2\le1\text{ a.e.}\bigr\}.
\]
The isometry theorem \cite[Theorem~5.8]{LenzSchmidtWirth2019}
therefore extends $\phi$ to the metric completions and gives
\[
 \frac{d_1(\phi(x),\phi(y))}{\sqrt{t_1}}
 =\frac{d_2(x,y)}{\sqrt{t_2}}\qquad(x,y\in X_2).
\]
The spaces are compact, hence complete. Conversely, a
measure-preserving isometry of these rescaled metrics preserves their
Cheeger forms; \cref{thm:dirichlet-canvas-reconstruction} gives equality
of laws.
\end{proof}

\begin{remark}[Why an arbitrary reference measure does not fix time]
\label{rem:rcd-time-scaling}
For $c>0$, keep $\mu$ fixed and replace $d$ by $cd$. Then
\[
 L_{cd,\mu}=c^{-2}L_{d,\mu},\qquad
 (Q_{cd,\mu,c^2t},C_{cd,\mu,c^2t})
 =(Q_{d,\mu,t},C_{d,\mu,t}).
\]
If $(X,d,\mu)$ is $\operatorname{RCD}(K,N)$, then $(X,cd,\mu)$ is
$\operatorname{RCD}(c^{-2}K,N)$. Hence $c\ne1$ gives the same law at
different times. On an $n$-dimensional closed manifold,
\[
 dv_{c^2g}=c^n dv_g,\qquad
 \frac{dv_{c^2g}}{\operatorname{Vol}_{c^2g}(M)}
 =\frac{dv_g}{\operatorname{Vol}_g(M)}.
\]
Thus normalized volume also permits this ambiguity; unnormalized volume
requires $c^n=1$.
More generally, if $\mu_i=\mathcal H^{n_i}_{d_i}$ with
$0<\mu_i(X_i)<\infty$ and $n_i>0$, the rescaled isometry forces
$n_1=n_2=n$ and
\[
 \mu_1(X_1)=\left(\frac{t_1}{t_2}\right)^{n/2}\mu_2(X_2).
\]
Measure preservation then gives $t_1=t_2$.
\end{remark}

\begin{remark}[Other Dirichlet forms]
The semigroup conclusion uses only (D1)--(D3). For quasi-regular forms,
realization by quasi-homeomorphisms is treated in
\cite{LenzSchmidtWirth2019}. For regular strongly local forms, a metric
conclusion requires the hypotheses of \cite[Theorem~5.8]{LenzSchmidtWirth2019},
including the Sobolev-to-Lipschitz property; compatibility of the intrinsic
distance with the topology alone does not suffice.
\end{remark}

\appendix
\section{Tensor categories and geometric automorphisms}
\label{sec:positive-protannakian-geometry}

Fix $t>0$ and a heat geometry $\mathfrak X$ as in
\eqref{eq:network-Laplace-type-datum}.

\subsection{Finite-dimensional tensor categories}

Write $H=H_{\mathfrak X}$, $Q=Q_{\mathfrak X,t}$, and
$C=C_{\mathfrak X,t}$. The support operator
\[
 D=L_QL_Q^*+L_CL_C^*
\]
is positive, trace class, and injective. For each regular threshold
$\delta>0$, meaning $\delta\notin\sigma(D)$ as in
\cref{subsec:spectral-truncations}, set
\[
 \begin{gathered}
 P_\delta=\mathbf1_{[\delta,\infty)}(D),\qquad
 H_\delta=P_\delta H,\qquad B_\delta=P_\delta DP_\delta,\\
 Q_\delta=P_\delta^{\odot2}Q,\qquad
 C_\delta=P_\delta^{\odot4}C.
 \end{gathered}
\]
Including $B_\delta$ forces every automorphism to preserve its
spectral subspaces.

Let $\mathsf D_{\mathrm{free}}$ be the free real linear rigid symmetric
monoidal $*$-category generated by one self-dual object $\mathsf h$, a
self-adjoint endomorphism $\mathsf b$, and symmetric arrows
$\mathsf t_j:\mathbf1\to\mathsf h^{\otimes j}$ for $j=2,4$.
Define $\omega^0_\delta$ on the generators by
\[
 \omega^0_\delta(\mathsf h)=H_\delta,\qquad
 \omega^0_\delta(\mathsf b)=B_\delta,\qquad
 \omega^0_\delta(\mathsf t_2)(1)=Q_\delta,\qquad
 \omega^0_\delta(\mathsf t_4)(1)=C_\delta.
\]
For an orthonormal basis $(e_i)_{i=1}^{d_\delta}$ of $H_\delta$, the
duality maps are evaluated as
\[
 \omega^0_\delta(\operatorname{ev})(x\otimes y)=\langle x,y\rangle,
 \qquad
 \omega^0_\delta(\operatorname{coev})(1)
       =\sum_{i=1}^{d_\delta}e_i\otimes e_i.
\]
Permutations exchange tensor factors, and
\[
 \omega^0_\delta(ba)=\omega^0_\delta(b)\omega^0_\delta(a),\qquad
 \omega^0_\delta(a\otimes b)=\omega^0_\delta(a)\otimes\omega^0_\delta(b),\qquad
 \omega^0_\delta(a^*)=\omega^0_\delta(a)^*.
\]
Quotient by the morphisms $a$ with $\omega^0_\delta(a)=0$, then split
the self-adjoint idempotents:
\[
 \mathsf D^{\mathrm{eff}}_\delta
 :=\operatorname{Kar}\bigl(
 \mathsf D_{\mathrm{free}}/\ker\omega^0_\delta\bigr),
 \qquad
 \omega_\delta:\mathsf D^{\mathrm{eff}}_\delta
 \longrightarrow\mathsf{Hilb}^{\mathbb R}_{\mathrm{fd}}.
\]
Explicitly, $\operatorname{Kar}$ has objects $(\mathsf x,p)$ with
$p^2=p=p^*$, morphisms
\[
 \operatorname{Hom}((\mathsf x,p),(\mathsf y,r))
 =r\operatorname{Hom}(\mathsf x,\mathsf y)p,
 \qquad
 \omega_\delta(\mathsf x,p)=\Ran\omega^0_\delta(p).
\]
The induced functor is faithful, symmetric monoidal and preserves adjoints.
For $d:\mathbf1\to\mathbf1$, set $\chi_\delta(d)=\omega_\delta(d)(1)$.
Choose $\operatorname{coev}_{\mathsf x}:\mathbf1\to\mathsf x^*\otimes\mathsf x$
and $\operatorname{ev}_{\mathsf x}:\mathsf x^*\otimes\mathsf x\to\mathbf1$.
For an endomorphism $a$ of $\mathsf x$, set
\[
 \begin{aligned}
 \operatorname{cl}(a)&=\operatorname{ev}_{\mathsf x}
       (\Id_{\mathsf x^*}\otimes a)\operatorname{coev}_{\mathsf x},\\
 \chi_\delta(\operatorname{cl}(a))&=\Tr\omega_\delta(a).
 \end{aligned}
\]
Thus
\[
 \chi_\delta(d\sqcup d')=\chi_\delta(d)\chi_\delta(d'),
 \qquad
 \chi_\delta\!\left(\operatorname{cl}(a^*a)\right)
 =\norm{\omega_\delta(a)}_{\mathrm{HS}}^2\ge0.
\]
For example,
\[
 \chi_\delta(\operatorname{cl}(\Id_{\mathsf h}))=d_\delta,
 \qquad
 \chi_\delta\bigl(\mathsf t_2^*(\mathsf b\otimes\mathsf b)\mathsf t_2\bigr)
 =\langle Q_\delta,(B_\delta\otimes B_\delta)Q_\delta\rangle.
\]

\begin{proposition}[Contractions determine the finite-dimensional object]
\label{prop:positive-cut-character-complete}
For each regular $\delta$, $\chi_\delta$ determines the orthogonal orbit of
\[
 (H_\delta,B_\delta,Q_\delta,C_\delta)
\]
and the pair $(\mathsf D^{\mathrm{eff}}_\delta,\omega_\delta)$ up to an
orthogonal symmetric monoidal equivalence preserving the generating
object and arrows. Conversely, such an equivalence preserves $\chi_\delta$.
\end{proposition}

\begin{proof}
The value $\chi_\delta(\operatorname{cl}(\Id_{\mathsf h}))$ gives
$\dim H_\delta$. Each complete contraction of
$(B_\delta,Q_\delta,C_\delta)$ is a composition of their generating
maps, tensor permutations and $\operatorname{ev}$, hence a value of
$\chi_\delta$.
\Cref{prop:finite-orthogonal-orbit-separation} determines its orbit.
An orthogonal transport intertwines the evaluation functors and hence
their quotients and idempotent completions. Moreover,
\[
 \omega^0_\delta(a)=0
 \quad\Longleftrightarrow\quad
 \chi_\delta\!\left(\operatorname{cl}(a^*a)\right)=0,
\]
so the character determines the evaluation kernel. The converse follows
by evaluating closed diagrams.
\end{proof}

Naturality on the generators identifies the orthogonal tensor
automorphism group with
\[
 G_\delta:=\operatorname{Aut}^{\otimes}(\omega_\delta)
 =\left\{U\in O(H_\delta):\begin{array}{l}
 UB_\delta U^*=B_\delta,\\
 U^{\odot2}Q_\delta=Q_\delta,\quad
 U^{\odot4}C_\delta=C_\delta
 \end{array}\right\}.
\]
For regular $0<\varepsilon<\delta$, define
\[
 G_\varepsilon\longrightarrow G_\delta,\qquad
 U\longmapsto U|_{H_\delta}.
\]
This is well-defined since $U$ commutes with the spectral projection
$P_\delta$ and
$Q_\delta=P_\delta^{\odot2}Q_\varepsilon$,
$C_\delta=P_\delta^{\odot4}C_\varepsilon$.

\begin{definition}[System of finite-dimensional contractions]
\label{def:positive-orthogonal-procharacter}
Consider the system
\[
 \chi_{\mathfrak X,t,\bullet}
 :=\bigl(H_\delta,\mathsf D^{\mathrm{eff}}_\delta,
 \omega_\delta,\chi_\delta,G_\varepsilon\to G_\delta\bigr)_\delta
\]
with its inclusions $H_\delta\subset H_\varepsilon$ for
$\varepsilon<\delta$. Two systems are isomorphic if, on a common cofinal
set of regular thresholds, there are orthogonal maps $U_\delta$ with
\[
 \begin{gathered}
 U_\delta B_\delta U_\delta^*=B'_\delta,\qquad
 U_\delta^{\odot2}Q_\delta=Q'_\delta,\qquad
 U_\delta^{\odot4}C_\delta=C'_\delta,\\
 U_\varepsilon|_{H_\delta}=U_\delta\quad(\varepsilon<\delta).
 \end{gathered}
\]
Equivalently, their generated
categories admit compatible orthogonal symmetric monoidal equivalences
preserving the generators and intertwining the fiber functors and
characters.
\end{definition}

\begin{proposition}[The stabilizer as a projective limit]
\label{prop:protannaka-limit-stabilizer}
Restriction to the spectral subspaces gives a topological group isomorphism
\[
 \Stab_t(\mathfrak X)
 \xrightarrow{\ \simeq\ }
 \varprojlim_{\delta\downarrow0}G_\delta,
 \qquad U\longmapsto(U|_{H_\delta})_\delta,
\]
where
\[
 \varprojlim_{\delta\downarrow0}G_\delta
 =\left\{(U_\delta)_\delta\in\prod_\delta G_\delta:
       U_\varepsilon|_{H_\delta}=U_\delta\text{ for }\varepsilon<\delta\right\}.
\]
Use the strong operator topology on the left and the subspace topology
from $\prod_\delta G_\delta$ on the right.
Here $\Stab_t(\mathfrak X)$ is the stabilizer
\eqref{eq:radial-heat-tensor-stabilizer}.
\end{proposition}

\begin{proof}
A stabilizer of $Q,C$ commutes with $D$, hence restricts to every
$G_\delta$. These restrictions determine it on the dense union
$\bigcup_\delta H_\delta$.

Conversely, a compatible family $(U_\delta)$ defines
$Ux=U_\delta x$ on $H_\delta$. It extends isometrically to $H$;
the compatible family $(U_\delta^{-1})$ gives its inverse. For $q=2,4$
and respectively $T=Q,C$,
\[
 U^{\odot q}T_\delta=T_\delta,\qquad
 \|U^{\odot q}T-T\|\le2\|T-T_\delta\|\longrightarrow0.
\]
Thus $U$ fixes both tensors. On finite-dimensional $H_\delta$, strong
and norm convergence agree. Conversely,
\[
 \|(U_i-U)x\|
 \le2\|x-P_\delta x\|
   +\|(U_i-U)|_{H_\delta}\|\,\|P_\delta x\|
\]
proves strong convergence on $H$ from convergence of all restrictions.
\end{proof}

Write $\operatorname{Aut}^{\otimes}(\omega_{\mathfrak X,t,\bullet})$
for this projective-limit group. For the tensor-category and fiber-functor
formalism, see \cite{Saavedra1972,DoplicherRoberts1989,
DeligneTensorCategories,EtingofTensorCategories2015}.

\subsection{Geometric morphisms}

Let $\mathsf{Heat}$ be the groupoid of heat geometries, with arrows
given by base isometries and
orthogonal bundle maps preserving the connection and potential.
The arrow $(\phi,J)$ from $\mathfrak X$ to $\mathfrak Y$ acts on
sections by $(Us)(y)=J_y s(\phi(y))$. Let $\mathsf{Ten}_{2,4}$ have the
triples $\mathcal T_t(\mathfrak X)$ as objects. Its arrows are the
orthogonal isomorphisms carrying both tensors of one triple to those
of another.

\begin{corollary}[Full faithfulness of the heat-tensor construction]
\label{thm:positive-scalar-tannaka-duality}
The functor
\[
 \mathcal T_t:\mathsf{Heat}\longrightarrow\mathsf{Ten}_{2,4}
\]
is fully faithful:
\[
 \operatorname{Hom}_{\mathsf{Heat}}(\mathfrak X,\mathfrak Y)
 \simeq
 \operatorname{Hom}_{\mathsf{Ten}_{2,4}}
       (\mathcal T_t(\mathfrak X),\mathcal T_t(\mathfrak Y)).
\]
The following data determine one another:
\[
 [\mathfrak X]_{\mathrm{isometry+gauge}},\qquad
 [\chi_{\mathfrak X,t,\bullet}]_\otimes,\qquad
 \Law_{C(\mathbb R)}(\mathcal Z_{\mathfrak X,t}).
\]
Moreover,
\[
 \operatorname{Aut}(\mathfrak X)
 \simeq\Stab_t(\mathfrak X)
 \simeq\operatorname{Aut}^{\otimes}
       (\omega_{\mathfrak X,t,\bullet}).
\]
\end{corollary}

\begin{proof}
\Cref{thm:scalar-observable-hears-gauge-field} gives the bijection on
morphisms and identifies the tensor orbit with the process law.
Orthogonal transports commute with the support operators and restrict
to compatible transports of their spectral truncations. Conversely,
such compatible maps extend by density, as in
\cref{prop:protannaka-limit-stabilizer}. The latter proposition also
identifies the automorphism groups.
\end{proof}

The assignment $(\mathfrak X,t)\mapsto\mathcal T_t(\mathfrak X)$ is
also fully faithful on positive-dimensional heat geometries with
morphisms required to preserve $t$.
Indeed, \cref{thm:scalar-observable-hears-gauge-field} makes every
tensor transport preserve the time and determines its unique geometric map.

\section{Homogeneous chaos realizations}
\label{sec:homogeneous-chaos-realizations}

We encode $(Q,C)$ by two tensors of any order $q\ge3$.
The construction preserves morphisms for odd $q$; for even $q$, its
only additional automorphism is reflection of an auxiliary line.

For a tensor $T$ of order $q$, put
\[
 \Sym T=\frac1{q!}\sum_{\pi\in\mathfrak S_q}\pi T,\qquad
 \pi(h_1\otimes\cdots\otimes h_q)
 =h_{\pi^{-1}(1)}\otimes\cdots\otimes h_{\pi^{-1}(q)}.
\]
For
$T\in H^{\odot2}$, the self-adjoint Hilbert--Schmidt operator $T^\flat$
is defined in \cref{sec:tensor-reconstruction}. For $C\in H^{\odot4}$,
define the self-adjoint Hilbert--Schmidt operator $\widehat C$ on
$H^{\odot2}$ by
\[
 \langle \widehat C\,\xi,\xi'\rangle_{H^{\odot2}}
 =\langle C,\xi\odot\xi'\rangle_{H^{\odot4}}.
\]
The dense span of the $\xi\odot\xi'$ makes $C\mapsto\widehat C$
injective.

\begin{definition}[Positive quadratic and quartic tensors]
\label{def:admissible-positive-master-packet}
Consider a triple
\[
 \mathbf T=(H,Q,C),
 \qquad Q\in H^{\odot2},\quad C\in H^{\odot4},
\]
where $H$ is a real separable Hilbert space, subject to
\begin{enumerate}[label=\textnormal{(\alph*)}]
\item $Q^\flat$ is positive and injective;
\item $\widehat C$ is positive and trace class on $H^{\odot2}$.
\end{enumerate}
An arrow
$U:(H,Q,C)\to(H',Q',C')$ is an orthogonal isomorphism satisfying
\[
 Q'=U^{\odot2}Q,
 \qquad
 C'=U^{\odot4}C.
\]
Let $\mathsf{Pos}_{2,4}$ be the groupoid of these triples and arrows.
\end{definition}

\begin{lemma}[Positivity of the quartic operator]
\label{lem:radial-quartic-positive-flattening}
Let $\eta\in H^{\odot2}$ with $\eta^\flat$ positive and trace class.
Then $\widehat{\eta\odot\eta}$ is positive and trace class, and for
every $\xi\in H^{\odot2}$,
\begin{equation}
 \left\langle\widehat{\eta\odot\eta}\,\xi,\xi\right\rangle
 =\frac13\left(
   \langle \eta,\xi\rangle^2
   +2\left\|(\eta^\flat)^{1/2}\xi^\flat(\eta^\flat)^{1/2}\right\|_{\mathcal S_2}^2
 \right)\geq0.
 \label{eq:radial-quartic-positive-formula}
\end{equation}
Moreover,
\begin{equation}
 \Tr\bigl(\widehat{\eta\odot\eta}\bigr)
 =\frac13\left((\Tr\eta^\flat)^2+2\norm{\eta^\flat}_{\mathcal S_2}^2\right).
 \label{eq:radial-quartic-trace-formula}
\end{equation}
\end{lemma}

\begin{proof}
Put $A=\eta^\flat$ and identify $H^{\odot2}$ with the self-adjoint
Hilbert--Schmidt operators. Symmetrization gives
\[
 \widehat{\eta\odot\eta}(X)
 =\frac13\bigl(\Tr(AX)A+2AXA\bigr).
\]
Taking the inner product with $X=\xi^\flat$ gives
\eqref{eq:radial-quartic-positive-formula}. If $(\lambda_i)$ are the
eigenvalues of $A$, the trace of $X\mapsto AXA$ is
\[
 \sum_i\lambda_i^2+\sum_{i<j}\lambda_i\lambda_j
 =\tfrac12\bigl((\Tr A)^2+\|A\|_{\mathcal S_2}^2\bigr).
\]
The rank-one term has trace $\|A\|_{\mathcal S_2}^2$, proving
\eqref{eq:radial-quartic-trace-formula}.
\end{proof}

\begin{lemma}[Heat tensors are admissible]
\label{lem:heat-tensors-admissible}
For every heat geometry $\mathfrak X$ and $t>0$, the triple
$\mathcal T_t(\mathfrak X)$ of \cref{def:master-heat-packet} belongs to
$\mathsf{Pos}_{2,4}$, and $(Q_{\mathfrak X,t},C_{\mathfrak X,t})$ has
joint support $H_{\mathfrak X}$.
\end{lemma}

\begin{proof}
The operator $Q^\flat=e^{-t\mathcal L}$ is positive and injective,
so the joint support is $H$. With the notation of
\eqref{eq:network-heat-evaluation-map},
$\eta_{t,x}^\flat=R_{t,x}^*R_{t,x}\ge0$ has finite rank and
\[
 \|\eta_{t,x}^\flat\|_{\mathcal S_2}
 \le\Tr\eta_{t,x}^\flat=\|R_{t,x}\|_{\mathcal S_2}^2.
\]
Thus \cref{lem:radial-quartic-positive-flattening} gives
\[
 \widehat C\ge0,\qquad
 \Tr\widehat C
 \le\int_M\|R_{t,x}\|_{\mathcal S_2}^4\,dv(x)
 \le\operatorname{Vol}(M)\sup_x\|R_{t,x}\|_{\mathcal S_2}^4<\infty.
\]
The supremum is finite by smoothness of the heat kernel.
\end{proof}

\subsection{Tensors of order \texorpdfstring{$q\geq4$}{q>=4}}
\label{subsec:anchor-padding}

Let $a$ be a fixed unit vector orthogonal to $H$.  For
$T\in H^{\odot k}$ and $q\geq k$, define
\begin{equation}
 \operatorname{Pad}_{q,k}^{a}(T)
 :=\binom qk\Sym\bigl(a^{\otimes(q-k)}\otimes T\bigr)
 \in(\mathbb Ra\oplus H)^{\odot q}.
 \label{eq:anchor-padding-definition}
\end{equation}
Equivalently, for $r\in\mathbb R$ and $h\in H$,
\[
 \left\langle\operatorname{Pad}_{q,k}^{a}(T),(ra+h)^{\otimes q}\right\rangle
 =\binom qk r^{q-k}\langle T,h^{\otimes k}\rangle.
\]
For $q\geq4$, put
\begin{equation}
 \begin{split}
 K_q(\mathbf T)&=\mathbb Ra\oplus H,\\
 P_q(\mathbf T)&=a^{\otimes q},\\
 S_q(\mathbf T)&=
 \operatorname{Pad}_{q,2}^{a}(Q)
 +\operatorname{Pad}_{q,4}^{a}(C).
 \end{split}
 \label{eq:q-homogeneous-padding-packet}
\end{equation}
Thus the two polynomials are
\[
 \begin{aligned}
 \langle P_q,(ra+h)^{\otimes q}\rangle&=r^q,\\
 \langle S_q,(ra+h)^{\otimes q}\rangle
 &=\binom q2r^{q-2}\langle Q,h^{\otimes2}\rangle
   +\binom q4r^{q-4}\langle C,h^{\otimes4}\rangle.
 \end{aligned}
\]
\begin{proposition}[Homogeneous realization in order at least four]
\label{thm:anchor-padding-faithfulness}
Let $q\ge4$. The assignment
\[
 \mathcal E_q:(H,Q,C)\longmapsto
 \bigl(K_q(\mathbf T),P_q(\mathbf T),S_q(\mathbf T)\bigr),
 \qquad
 U\longmapsto\Id_{\mathbb Ra}\oplus U,
\]
is a functor from $\mathsf{Pos}_{2,4}$ to the groupoid of ordered pairs
of symmetric tensors of order $q$ and their orthogonal transports. The
pair $(P_q,S_q)$ has joint support $K_q(\mathbf T)$, and two such pairs
are orthogonally equivalent if and only if the original pairs $(Q,C)$
are. On arrows $W$ satisfying $Wa=a'$, the functor $\mathcal E_q$ is
fully faithful. For odd $q$, every arrow transporting $P_q$ satisfies
$Wa=a'$. For even $q$, restriction $W\mapsto W|_H$ is surjective on
each morphism set and has automorphism kernel
$\{\Id,(-\Id_{\mathbb Ra})\oplus\Id_H\}$; quotienting by this kernel
gives an equivalent groupoid.
\end{proposition}

\begin{proof}
Every transport $U$ of $(Q,C)$ extends by fixing $a$.
Conversely, a transport $W$ of $(P_q,S_q)$ satisfies
$Wa=\varepsilon a'$ with $\varepsilon\in\{\pm1\}$ and
$\varepsilon^q=1$. Put $U=W|_{a^\perp}:H\to H'$.
Comparing the coefficients of $r^{q-2}$ and $r^{q-4}$ gives
\[
 Q'=\varepsilon^{q-2}U^{\odot2}Q,
 \qquad C'=\varepsilon^{q-4}U^{\odot4}C.
\]
Both signs are $+1$: either $q$ is odd and $\varepsilon=1$, or
$q$ is even. Hence every transport is the unique extension of $U$
fixing $a$, or, when $q$ is even, its composition with reflection
of the auxiliary line. This proves all assertions on morphisms.

The joint support contains $a$ by $P_q=a^{\otimes q}$. Moreover,
\[
 \langle S_q,a^{\otimes(q-2)}\otimes h\otimes k\rangle
 =\langle Q,h\otimes k\rangle\qquad(h,k\in H),
\]
so it contains $\overline{\Ran Q^\flat}=H$.
\end{proof}

\subsection{A cubic realization}
\label{subsec:minimal-positive-cubification}

The cubic construction uses the positive square root of $\widehat C$.
Let
\[
 \mathcal H_C=\overline{\Ran\widehat C^{1/2}}\subset H^{\odot2},
 \qquad
 B_C:H^{\odot2}\longrightarrow\mathcal H_C,\qquad
 B_C\xi=\widehat C^{1/2}\xi.
\]
Then $B_C$ is Hilbert--Schmidt and has dense range, and
\[
 B_C^*B_C=\widehat C,
 \qquad
 \norm{B_C}_{\mathcal S_2}^2=\Tr(\widehat C).
\]
If $\widetilde B^*\widetilde B=\widehat C$ and $\widetilde B$ has dense
range, define $J(B_C\xi)=\widetilde B\xi$. The identity
\[
 \|B_C\xi\|^2=\langle\widehat C\xi,\xi\rangle
             =\|\widetilde B\xi\|^2
\]
makes $J$ well-defined and isometric. Density gives its unique orthogonal
extension and $\widetilde B=JB_C$.

Let $\beta_{B_C}\in \mathcal H_C\otimes H^{\odot2}$ be the Hilbert tensor
corresponding to the Hilbert--Schmidt map $B_C$:
\begin{equation}
 \langle\beta_{B_C},\ell\otimes\xi\rangle
 =\langle B_C\xi,\ell\rangle_{\mathcal H_C},
 \qquad \ell\in \mathcal H_C,\quad\xi\in H^{\odot2}.
 \label{eq:cubic-beta-definition}
\end{equation}
Identify the positive trace-class operator $B_CB_C^*$ on $\mathcal H_C$ with its
quadratic tensor.  On
\[
 K_3(\mathbf T)=\mathbb Ra\oplus H\oplus \mathcal H_C
\]
define
\begin{equation}
 \begin{split}
 R_{\mathbf T}&=Q\oplus(-B_CB_C^*)
       \in(H\oplus \mathcal H_C)^{\odot2},\\
 P_3(\mathbf T)&=a^{\otimes3},\\
 S_3(\mathbf T)&=
 3\Sym(a\otimes R_{\mathbf T})
 +3\Sym(\beta_{B_C})
       \in K_3(\mathbf T)^{\odot3}.
 \end{split}
 \label{eq:positive-cubification}
\end{equation}
For $x=ra+h+\ell$, with $h\in H$ and $\ell\in\mathcal H_C$, these
definitions read
\[
 \begin{aligned}
 \langle P_3,x^{\otimes3}\rangle&=r^3,\\
 \langle S_3,x^{\otimes3}\rangle
 &=3r\bigl(\langle Q,h^{\otimes2}\rangle-\|B_C^*\ell\|^2\bigr)
   +3\langle B_C(h^{\otimes2}),\ell\rangle.
 \end{aligned}
\]
In particular, for $z,z'\in H\oplus\mathcal H_C$ and $h,k\in H$,
\[
 \langle S_3,a\otimes z\otimes z'\rangle
       =\langle R_{\mathbf T},z\otimes z'\rangle,
 \qquad
 \langle S_3,\ell\otimes h\otimes k\rangle
       =\langle B_C(h\odot k),\ell\rangle.
\]

\begin{theorem}[Full faithfulness of the cubic realization]
\label{thm:positive-cubification-full-faithfulness}
The pair $(P_3(\mathbf T),S_3(\mathbf T))$ has full joint support
$K_3(\mathbf T)$, and the assignment
\[
 \mathcal E_3:(H,Q,C)\longmapsto
 \bigl(K_3(\mathbf T),P_3(\mathbf T),S_3(\mathbf T)\bigr)
\]
is fully faithful from $\mathsf{Pos}_{2,4}$ to the groupoid of ordered
pairs of symmetric cubic tensors and their orthogonal transports.  Thus
\begin{equation}
 \begin{split}
 &\left\{U:H\to H':Q'=U^{\odot2}Q,
                    \ C'=U^{\odot4}C\right\}\\
 &\quad\xrightarrow{\ \simeq\ }
 \left\{W:K_3(\mathbf T)\to K_3(\mathbf T'):
 \begin{array}{l}
 P_3(\mathbf T')=W^{\odot3}P_3(\mathbf T),\\
 S_3(\mathbf T')=W^{\odot3}S_3(\mathbf T)
 \end{array}\right\}.
 \end{split}
 \label{eq:cubification-hom-bijection}
\end{equation}
All maps $U$ and $W$ in \eqref{eq:cubification-hom-bijection} are
orthogonal isomorphisms.  The map associated with $U$ is
\[
 \widetilde U=\Id_{\mathbb Ra}\oplus U\oplus
       \left.U^{\odot2}\right|_{\mathcal H_C}.
\]
Hence orthogonal equivalence of the cubic pairs is equivalent to
orthogonal equivalence of $(Q,C)$, and their stabilizers are isomorphic.
\end{theorem}

\begin{proof}
Let $U$ transport $(Q,C)$ and set $U_2=U^{\odot2}$. Then
\[
 \widehat C'=U_2\widehat C U_2^*,\qquad
 \widehat C'^{1/2}=U_2\widehat C^{1/2}U_2^*.
\]
Consequently,
\begin{equation}
 V_U:=U_2|_{\mathcal H_C}:\mathcal H_C\longrightarrow\mathcal H_{C'},
 \qquad B_{C'}U_2=V_UB_C,
 \label{eq:cubic-B-equivariance}
\end{equation}
where $V_U$ is orthogonal. Hence
\[
 \beta_{B_{C'}}=(V_U\otimes U_2)\beta_{B_C},\qquad
 B_{C'}B_{C'}^*=V_UB_CB_C^*V_U^*.
\]
Substitution in \eqref{eq:positive-cubification} shows that
$\Id\oplus U\oplus V_U$ transports both cubics.

Conversely, a transport $W$ satisfies
$W^{\odot3}a^{\otimes3}=a'^{\otimes3}$, hence $Wa=a'$.
Write $W=\Id\oplus W_0$. Define
$\langle\iota_aS_3,z\otimes z'\rangle
 :=\langle S_3,a\otimes z\otimes z'\rangle$. The first restriction gives
\begin{equation}
 \iota_aS_3=R_{\mathbf T},
 \label{eq:cubic-R-extraction}
\end{equation}
and therefore
\begin{equation}
 R_{\mathbf T}^\flat=Q^\flat\oplus(-B_CB_C^*),\qquad
 W_0R_{\mathbf T}^\flat W_0^*=R_{\mathbf T'}^\flat.
 \label{eq:cubic-signed-operator}
\end{equation}
Since $Q^\flat$ is injective and $B_C$ has dense range, the positive
and negative spectral subspaces are exactly $H$ and $\mathcal H_C$.
Thus $W_0=U\oplus V$, and $Q'=U^{\odot2}Q$.
The second restriction gives
\begin{equation}
 \beta_{B_{C'}}=(V\otimes U^{\odot2})\beta_{B_C},\qquad
 B_{C'}U^{\odot2}=VB_C.
 \label{eq:cubic-B-recovery}
\end{equation}
Taking Gram operators yields
\[
 \widehat C'=U^{\odot2}\widehat C(U^{\odot2})^*,
 \qquad C'=U^{\odot4}C.
\]
Equations \eqref{eq:cubic-B-equivariance} and
\eqref{eq:cubic-B-recovery} imply $V=V_U$ on the dense range of
$B_C$, hence everywhere. This proves uniqueness of the extension.

The joint support contains $a$ and, by
\eqref{eq:cubic-R-extraction}, $\Ran R_{\mathbf T}^\flat$.
The latter is dense in $H\oplus\mathcal H_C$, proving full support.
\end{proof}

\subsection{A scalar process in each homogeneous chaos}
\label{subsec:pure-chaos-scalar-law}

\begin{lemma}[Several tensors of equal order]
\label{lem:typed-homogeneous-canvas}
Fix $q\geq1$ and distinct real numbers $s_1,\ldots,s_r$.  Let
$\mathbf f=(f_1,\ldots,f_r)$ with $f_\alpha\in H^{\odot q}$, and let
$(\mathbb W_u)_{u\in\mathbb R}$ be the centered Gaussian family with
\[
 \mathbb E[\mathbb W_u(h)\mathbb W_v(k)]
 =e^{-(u-v)^2}\langle h,k\rangle_H.
\]
Define the real process
\begin{equation}
 \mathcal Z_{\mathbf f}^{(q)}(u)
 :=\sum_{\alpha=1}^r e^{s_\alpha u}
       I_q^{\mathbb W_u}(f_\alpha).
 \label{eq:typed-pure-chaos-process}
\end{equation}
The process has a continuous version and takes values in the $q$-th
chaos of the joint Gaussian family. It need not be stationary.
For a second family $\mathbf f'$ with the same weights,
\begin{equation}
 \begin{split}
 &\Law_{C(\mathbb R)}(\mathcal Z_{\mathbf f}^{(q)})
 =\Law_{C(\mathbb R)}(\mathcal Z_{\mathbf f'}^{(q)})\\
 &\quad\Longleftrightarrow\quad
 \exists\,U:S(\mathbf f)\to S(\mathbf f')\text{ orthogonal isomorphism such that }
 f_\alpha'=U^{\odot q}f_\alpha\ \text{for every }\alpha.
 \end{split}
 \label{eq:typed-canvas-orbit-equivalence}
\end{equation}
Here $S(\mathbf f)$ is the closed joint support of the tensors.
\end{lemma}

\begin{proof}
Orthogonal transports preserve the process law. For $n$ factors, let
$N$ run over symmetric nonnegative integer matrices satisfying
$N_{ii}=0$ and $\sum_jN_{ij}=q$, and let
$\tau\in\{1,\ldots,r\}^n$. Put
\begin{equation}
 \chi_{N,\tau}(u)=\exp\left(\sum_i s_{\tau(i)}u_i
       -\sum_{i<j}N_{ij}(u_i-u_j)^2\right).
 \label{eq:typed-canvas-coefficient}
\end{equation}
Wick's formula is
\[
 \E\prod_{i=1}^n\mathcal Z_{\mathbf f}^{(q)}(u_i)
 =\sum_{N,\tau}\frac{(q!)^n}{\prod_{i<j}N_{ij}!}
       \operatorname{Contr}_N(f_{\tau(1)},\ldots,f_{\tau(n)})
       \chi_{N,\tau}(u).
\]
Write $Q_N(w)=\sum_{i<j}N_{ij}(w_i-w_j)^2$ and
$\ell_\tau(w)=\sum_i s_{\tau(i)}w_i$.
Their coefficients determine $N$ and $\tau$. For generic $w$, the
finitely many pairs $(Q_N(w),\ell_\tau(w))$ in a given moment are
distinct. On $u=\theta w$, the corresponding functions are
$e^{-Q_N(w)\theta^2+\ell_\tau(w)\theta}$. They are linearly independent:
as $\theta\to+\infty$, the smallest quadratic coefficient, then the
largest linear coefficient, selects one term. Repeating recovers all
coefficients. Equality of laws thus gives every labelled contraction;
all moments exist by hypercontractivity.

Apply the spectral-truncation proof of
\cref{thm:scalar-replica-tannaka} with
$D_{\mathbf f}=\sum_\alpha L_{f_\alpha}L_{f_\alpha}^*$.
It recovers all contractions of
\[
 \left(P_\delta D_{\mathbf f}P_\delta,
       (P_\delta^{\odot q}f_\alpha)_{\alpha=1}^r\right),
 \qquad P_\delta=\mathbf1_{[\delta,\infty)}(D_{\mathbf f}).
\]
\Cref{prop:finite-orthogonal-orbit-separation} gives the finite-dimensional
transports. The compact inverse-limit argument gives compatible choices,
whose union extends to the required orthogonal map of joint supports.

In \eqref{eq:canvas-single-isonormal-realization}, the deterministic
kernel is the smooth function
\[
 F(u)=\sum_\alpha e^{s_\alpha u}
       \bigl(\Psi(u)^{\otimes q}\otimes f_\alpha\bigr).
\]
For $u,v$ in a compact interval $I$ and $p\ge2$,
\[
 \|\mathcal Z_{\mathbf f}^{(q)}(u)-\mathcal Z_{\mathbf f}^{(q)}(v)\|_p
 \le(p-1)^{q/2}\sqrt{q!}\sup_I\|F'\|\,|u-v|.
\]
Kolmogorov's criterion proves continuity.
\end{proof}

For the pairs $(P_q,S_q)$ defined in
\eqref{eq:q-homogeneous-padding-packet} and
\eqref{eq:positive-cubification}, the weights $s_P=0$ and $s_S=1$ give
the process
\[
 \mathcal Z_{\mathbf T}^{(q)}(u)
 =I_q^{\mathbb W_u}\bigl(P_q(\mathbf T)\bigr)
 +e^u I_q^{\mathbb W_u}\bigl(S_q(\mathbf T)\bigr),
 \qquad q\geq3.
\]
\begin{corollary}[Homogeneous chaos processes]
\label{cor:all-homogeneous-chaoses-hear-master-packet}
Let $\mathbf T,\mathbf T'\in\mathsf{Pos}_{2,4}$.
For every $q\geq3$,
\[
 \begin{split}
 &\Law_{C(\mathbb R)}(\mathcal Z_{\mathbf T}^{(q)})
 =\Law_{C(\mathbb R)}(\mathcal Z_{\mathbf T'}^{(q)})\\
 &\quad\Longleftrightarrow\quad
 \exists\,U:H\to H'\text{ satisfying }
 Q'=U^{\odot2}Q,\qquad C'=U^{\odot4}C.
 \end{split}
\]
For the heat tensors of \cref{def:master-heat-packet}, these
conditions are further equivalent to equivalence of the underlying
manifold, metric, Euclidean bundle, metric connection, and self-adjoint
potential.
\end{corollary}

\begin{proof}
The padded and cubic pairs have full support, and their orbits determine
that of $(Q,C)$ by
\cref{thm:anchor-padding-faithfulness,thm:positive-cubification-full-faithfulness}.
Apply \cref{lem:typed-homogeneous-canvas} with weights $0,1$.
For heat tensors, use
\cref{lem:heat-tensors-admissible,thm:scalar-observable-hears-gauge-field}.
\end{proof}

\section{Instability and inverse continuity on compact families}
\label{sec:stability-sharpness}

Convergence of the heat data need not control curvature. On a compact
smooth family, reconstruction nevertheless has a continuous inverse.

\subsection{Oscillating connections}

On the flat torus $M=\mathbb R^2/(2\pi\mathbb Z)^2$, take
$E=M\times\mathbb R^2$ and $J_0=\left(\begin{smallmatrix}0&-1\\1&0\end{smallmatrix}\right)$.
Set $\nabla^{(0)}=d$ and, for $k\ge1$,
\begin{equation}
 \nabla^{(k)}=d+a_k(x)J_0\,dy,\qquad a_k(x)=k^{-1}\sin(kx).
 \label{eq:oscillating-connection}
\end{equation}
Then $F_{\nabla^{(k)}}=\cos(kx)J_0\,dx\wedge dy$ and
\begin{equation}
 \norm{F_{\nabla^{(k)}}}_{L^\infty}=c_0,\qquad
 \int_M|F_{\nabla^{(k)}}|^2\,dv
 =\tfrac12\operatorname{Vol}(M)c_0^2,
 \label{eq:oscillating-curvature-size}
\end{equation}
where $c_0=|J_0|>0$ for a fixed orthogonally invariant norm.
Both quantities are gauge invariant.

For $t>0$, let $Q_{k,t},C_{k,t}$ be the heat tensors of
$\nabla^{(k)}$ and, on a common Gaussian family, put
\[
 \mathcal Z_{k,t}(u)
 =I_2^{\mathbb W_u}(Q_{k,t})+I_4^{\mathbb W_u}(C_{k,t}).
\]
Equip $C(\mathbb R)$ with
\[
 \rho_{\mathrm{loc}}(z,z')
 =\sum_{j\ge1}2^{-j}\min\{1,\|z-z'\|_{C([-j,j])}\}.
\]
Write $d_{\mathrm{BL}}$ for the supremum of
$|\int F\,d\mu-\int F\,d\nu|$ over $\|F\|_\infty\le1$ and
$\operatorname{Lip}_{\rho_{\mathrm{loc}}}(F)\le1$.

\begin{proposition}[Convergent heat data with nonconvergent curvatures]
\label{prop:fixed-time-instability}
For every $t>0$, on $H=L^2(M,\mathbb R^2)$,
\[
 Q_{k,t}\to Q_{0,t}\quad\text{in }H^{\odot2},\qquad
 C_{k,t}\to C_{0,t}\quad\text{in }H^{\odot4}.
\]
Consequently,
\begin{equation}
 \Law_{C(\mathbb R)}(\mathcal Z_{k,t})
 \longrightarrow\Law_{C(\mathbb R)}(\mathcal Z_{0,t})
 \label{eq:oscillating-law-convergence}
\end{equation}
weakly, although $\|F_{\nabla^{(k)}}\|_\infty=c_0$ and
$\nabla^{(0)}$ is flat. More precisely, for $p>2$, $R>0$,
\begin{equation}
 \E\|\mathcal Z_{k,t}-\mathcal Z_{0,t}\|_{C([-R,R])}^p\to0,
 \label{eq:oscillating-continuous-process-convergence}
\end{equation}
and, for $0<t_0<t_1<\infty$,
\begin{equation}
 \sup_{t\in[t_0,t_1]}d_{\mathrm{BL}}\bigl(
 \Law(\mathcal Z_{k,t}),\Law(\mathcal Z_{0,t})\bigr)\to0.
 \label{eq:oscillating-locally-uniform-time-convergence}
\end{equation}
\end{proposition}

\begin{proof}
Writing $L_k=(\nabla^{(k)})^*\nabla^{(k)}$, we have
\[
 L_k=L_0-2a_kJ_0\partial_y+a_k^2\Id,\qquad
 L_0=-\partial_x^2-\partial_y^2.
\]
Since $\|a_k\|_\infty\le k^{-1}$,
\begin{equation}
 \|(L_k-L_0)e^{-sL_0}\|_{2\to2}
 \le Ck^{-1}s^{-1/2}+k^{-2}.
 \label{eq:oscillating-relative-bound}
\end{equation}
Kato's inequality gives $|e^{-sL_k}f|\le e^{-s\Delta}|f|$.
The scalar heat bound in dimension two therefore yields
$\|e^{-sL_k}\|_{2\to\infty}\le C(1+s^{-1/2})$, uniformly in $k$.
Duhamel's formula reads
\[
 e^{-TL_k}-e^{-TL_0}
 =-\int_0^T e^{-(T-s)L_k}(L_k-L_0)e^{-sL_0}\,ds.
\]
Therefore
\begin{align}
 \|e^{-TL_k}-e^{-TL_0}\|_{2\to\infty}
 &\le\int_0^T\|e^{-(T-s)L_k}\|_{2\to\infty}
       \|(L_k-L_0)e^{-sL_0}\|_{2\to2}\,ds\notag\\
 &\le C_Tk^{-1}.
 \label{eq:oscillating-heat-kernel-convergence}
\end{align}
The endpoint singularities are bounded by a constant times
$(T-s)^{-1/2}s^{-1/2}$ and are integrable. Proving the formula first
on smooth sections and then using density justifies the operator bound.

Let $K_{k,t,x}f=(e^{-tL_k/2}f)(x)$. Smooth kernels turn the last
$L^\infty$ estimate into a uniform pointwise bound. Since the target
is $\mathbb R^2$,
\[
 \sup_x\|K_{k,t,x}-K_{0,t,x}\|_{\mathrm{HS}}\le C_tk^{-1},
 \qquad \sup_{k,x}\|K_{k,t,x}\|_{\mathrm{HS}}<\infty.
\]
With $\eta_{k,t,x}=\sum_{\alpha=1}^2
(K_{k,t,x}^*e_\alpha)^{\otimes2}$, the inequality
\[
 \|\eta_{k,t,x}-\eta_{0,t,x}\|
 \le(\|K_{k,t,x}\|_{\mathrm{HS}}+\|K_{0,t,x}\|_{\mathrm{HS}})
       \|K_{k,t,x}-K_{0,t,x}\|_{\mathrm{HS}}
\]
gives
\begin{equation}
 \sup_x\|\eta_{k,t,x}-\eta_{0,t,x}\|_{H^{\otimes2}}
 \le C_tk^{-1}.
 \label{eq:oscillating-pointwise-tensor-convergence}
\end{equation}
Integrate this bound and
\[
 \|\eta\odot\eta-\eta'\odot\eta'\|
 \le(\|\eta\|+\|\eta'\|)\|\eta-\eta'\|
\]
to obtain convergence of $Q_{k,t}$ and $C_{k,t}$.
\Cref{lem:continuous-universal-canvas} then proves
\eqref{eq:oscillating-continuous-process-convergence} and
\eqref{eq:oscillating-law-convergence}.

All bounds are uniform for $t\in[t_0,t_1]$. Applying
\eqref{eq:canvas-quantitative-continuity} on each $[-j,j]$ and
dominating the resulting series by $\sum_j2^{-j}$ gives
\[
 \sup_{t\in[t_0,t_1]}\E\rho_{\mathrm{loc}}
      (\mathcal Z_{k,t},\mathcal Z_{0,t})\to0.
\]
Now use $d_{\mathrm{BL}}(\Law(Y),\Law(Y'))
\le\E\rho_{\mathrm{loc}}(Y,Y')$ and
\eqref{eq:oscillating-curvature-size}.
\end{proof}

\begin{remark}[Failure of an inverse curvature modulus]
\label{cor:no-fixed-time-curvature-modulus}
For any metric $d_{\mathrm{law}}$ inducing weak convergence,
\cref{prop:fixed-time-instability} rules out
\[
 \|F_\nabla\|_\infty\le\omega\bigl(d_{\mathrm{law}}
 (\Law(\mathcal Z_{\nabla,t}),\Law(\mathcal Z_{d,t}))\bigr),
 \qquad\omega(s)\to0\quad(s\downarrow0),
\]
on the class \eqref{eq:oscillating-connection}: the two sides tend to
$c_0$ and $0$. The same obstruction holds for metrics inducing the
compact-open topology of $t\mapsto\Law(\mathcal Z_{\nabla,t})$ on
$t>0$, with $d_{\mathrm{BL}}$ on the target.
\end{remark}

\subsection{Inverse continuity on compact geometric classes}

\begin{definition}[Smooth convergence modulo diffeomorphism and gauge]
\label{def:smooth-CG-gauge-topology}
Write $[\mathfrak X_j]\to[\mathfrak X]$ smoothly if, eventually,
the base dimensions and bundle ranks agree and there are
diffeomorphisms $\phi_j:M\to M_j$ and smooth orthogonal bundle maps
$J_j:E\to\phi_j^*E_j$ for which
\[
 \phi_j^*g_j\to g,\qquad
 J_j^{-1}\phi_j^*\nabla^jJ_j\to\nabla,\qquad
 J_j^{-1}\phi_j^*V_jJ_j\to V
\]
in every $C^m$ norm. We use this smooth quotient topology on
isometry--gauge classes.
\end{definition}

For a compact metrizable family $\mathscr K$ in this topology, set
\[
 \mathscr S_t:\mathscr K\to\mathcal P(C(\mathbb R)),\qquad
 \mathscr S_t[\mathfrak X]=\Law(\mathcal Z_{\mathfrak X,t}),
 \qquad t>0.
\]

\begin{proposition}[Inverse continuity on a compact family]
\label{prop:compactness-uniqueness-stability}
The map $\mathscr S_t$ is a homeomorphism onto its compact image.
For any compatible metrics $d_{\mathscr K},d_{\mathrm{law}}$, there
is a nondecreasing function $\omega_{\mathscr K,t}$, tending to zero
at zero, such that
\begin{equation}
 d_{\mathscr K}(X,Y)\le\omega_{\mathscr K,t}
 \bigl(d_{\mathrm{law}}(\mathscr S_tX,\mathscr S_tY)\bigr)
 \qquad(X,Y\in\mathscr K).
 \label{eq:compactness-uniqueness-modulus}
\end{equation}
\end{proposition}

\begin{proof}
For a smoothly convergent sequence, pull back to $M$ and put
\[
 r_j=\frac{dv_{\phi_j^*g_j}}{dv_g},\qquad
 \mathcal U_jf=r_j^{-1/2}J_jf.
\]
Then $r_j\to1$ smoothly and
$\|\mathcal U_jf\|_{L^2(dv_{\phi_j^*g_j})}=\|f\|_{L^2(dv_g)}$.
Write $L_j$ for the pulled-back operator and put
\[
 \widetilde L_j=\mathcal U_j^{-1}L_j\mathcal U_j,\qquad
 \widetilde Q_{j,t}=(\mathcal U_j^{-1})^{\odot2}Q_{j,t},\qquad
 \widetilde C_{j,t}=(\mathcal U_j^{-1})^{\odot4}C_{j,t}.
\]
The transported tensors satisfy
\[
 \widetilde Q_{j,t}^{\flat}=e^{-t\widetilde L_j},\qquad
 \langle\widetilde C_{j,t},f^{\otimes4}\rangle
 =\int_M r_j^{-1}|e^{-t\widetilde L_j/2}f|^4\,dv_g.
\]
Here $r_j^{-1}=r_j^{-2}r_j$ comes from the fourth power of
$\mathcal U_jf$ and the volume density. The coefficients of
$\widetilde L_j$ converge smoothly. The heat parametrix and remainder
estimates give smooth convergence of the kernels at positive time
\cite{BerlineGetzlerVergne1992}; the defining tensor integrals then give
$\widetilde Q_{j,t}\to Q_t$ and $\widetilde C_{j,t}\to C_t$ in Hilbert
norm. Equation \eqref{eq:canvas-quantitative-continuity} proves continuity.

It is injective by \cref{thm:scalar-observable-hears-gauge-field}.
Compactness of $\mathscr K$ and the Hausdorff property of the target
give the homeomorphism assertion. The function
\[
 \omega_{\mathscr K,t}(s)=\sup\{d_{\mathscr K}(X,Y):X,Y\in\mathscr K,
 d_{\mathrm{law}}(\mathscr S_tX,\mathscr S_tY)\le s\}
\]
is nondecreasing. If it did not tend to zero, compactness would give
subsequences $X_j\to X$, $Y_j\to Y$ with
\[
 d_{\mathrm{law}}(\mathscr S_tX_j,\mathscr S_tY_j)\to0,
 \qquad d_{\mathscr K}(X_j,Y_j)\ge\varepsilon>0.
\]
Continuity gives $\mathscr S_tX=\mathscr S_tY$; injectivity gives
$X=Y$, contradicting the second inequality.
\end{proof}

\begin{corollary}[Joint recovery of geometry and time on compact families]
\label{cor:joint-geometry-time-continuity}
Let $\mathscr K$ be a compact metrizable family of positive-dimensional
heat geometries in the smooth quotient topology, and let
$0<a\le b<\infty$. Then
\[
 \mathscr S:\mathscr K\times[a,b]\longrightarrow\mathcal P(C(\mathbb R)),
 \qquad \mathscr S([\mathfrak X],t)=\Law(\mathcal Z_{\mathfrak X,t}),
\]
is a homeomorphism onto its compact image. Its inverse is uniformly
continuous for any compatible metrics on the product and on the image.
\end{corollary}

\begin{proof}
In the common $L^2$ realization used above, smooth convergence of the
geometric coefficients gives smooth convergence of the heat kernels,
locally uniformly in positive time. Thus convergence of both geometry
and time implies convergence of the two heat tensors and, by
\eqref{eq:canvas-quantitative-continuity}, of the process laws.
Injectivity follows from
\cref{thm:scalar-observable-hears-gauge-field}; compactness gives the
remaining assertions.
\end{proof}

\begin{remark}[Consistency from independent path observations]
\label{cor:compact-minimum-distance-consistency}
For independent paths $Z_j$ of law $\mathscr S_tX_0$, $X_0\in\mathscr K$,
Varadarajan's theorem \cite{Varadarajan1958} gives
$\widehat\nu_n=n^{-1}\sum_{j=1}^n\delta_{Z_j}\to\mathscr S_tX_0$
weakly almost surely. If $\mathscr K$-valued estimators satisfy
\[
 d_{\mathrm{law}}(\widehat\nu_n,\mathscr S_t\widehat X_n)
 \le\inf_{X\in\mathscr K}d_{\mathrm{law}}
       (\widehat\nu_n,\mathscr S_tX)+n^{-1},
\]
then
\[
 d_{\mathrm{law}}(\mathscr S_t\widehat X_n,\mathscr S_tX_0)
 \le2d_{\mathrm{law}}(\widehat\nu_n,\mathscr S_tX_0)+n^{-1}\to0.
\]
Thus $d_{\mathscr K}(\widehat X_n,X_0)\to0$ almost surely.
Each observation here is a complete path in the Polish space
$C(\mathbb R)$.
\end{remark}

\begin{remark}[Dependence on the topology of the data]
Uniformity on compact positive-time intervals still does not control
curvature in \cref{prop:fixed-time-instability}. Curvature estimates
require stronger geometric assumptions or stronger data, potentially
including control as $t\downarrow0$. The modulus in
\cref{prop:compactness-uniqueness-stability} is qualitative and depends
on the compact smooth family.
\end{remark}

\section{Finite tests for fixed candidate models}
\label{subsec:finite-statistical-identification}

Put $c_B(y)=\max\{-B,\min\{B,y\}\}$, $B>0$.

\begin{proposition}[Finite candidate classes]
\label{thm:finite-coordinate-statistical-test}
\label{cor:finite-candidate-statistical-identification}
Let $\mathbf f^{(1)},\ldots,\mathbf f^{(k)}$, $k\ge2$, be finite
tensor families, pairwise orthogonally inequivalent on their joint
supports. There are $P\le\binom k2$, integers $r_\ell\ge1$, rational
points $u_{\ell a}$ and numbers $B_\ell>0$ such that the statistics
\[
 \psi_\ell(z)=\frac1{B_\ell}c_{B_\ell}
       \left(\prod_{a=1}^{r_\ell}z(u_{\ell a})\right),
 \qquad \psi_\ell:C(\mathbb R)\to[-1,1],
\]
are bounded and continuous, and
\[
 v_i=\bigl(\E\psi_\ell(\mathcal Z_{\mathbf f^{(i)}})\bigr)_{\ell=1}^P,
 \qquad d_*:=\min_{i\ne j}\|v_i-v_j\|_\infty>0.
\]
For independent observations $Z_1,\ldots,Z_n$, define
\[
 \widehat v_{n,\ell}=\frac1n\sum_{j=1}^n\psi_\ell(Z_j),\qquad
 \widehat i_n\in\operatorname*{argmin}_{1\le i\le k}
                      \|\widehat v_n-v_i\|_\infty.
\]
Break ties arbitrarily. Under model $i$,
\begin{equation}
 \mathbb P_i(\widehat i_n\ne i)
 \le2P\exp\bigl(-nd_*^2/8\bigr).
 \label{eq:statistical-finite-class-bound}
\end{equation}
The test uses only the defining evaluations and is almost surely
eventually correct. For $k=2$, one may take
$\psi_1(Z)=c_B(\prod_{a=1}^rZ(u_a))/B$. If $\delta$ is the absolute
difference of the unnormalized clipped means, the error bound is
$2\exp(-n\delta^2/(8B^2))$.
\end{proposition}

\begin{proof}
For each pair, \cref{cor:finite-canvas-witness} supplies an integrable
product $Y(Z)=\prod_aZ(u_a)$ at rational points with distinct means.
Dominated convergence gives $\E c_B(Y)\to\E Y$; hence some finite
$B$ preserves the distinction. Choose one statistic $c_B(Y)/B$ for
each pair.

An error implies $\|\widehat v_n-v_i\|_\infty\ge d_*/2$.
Hoeffding's inequality \cite{Hoeffding1963}, applied to variables in
$[-1,1]$, gives
\[
 \mathbb P_i\bigl(|\widehat v_{n,\ell}-v_{i,\ell}|\ge d_*/2\bigr)
 \le2e^{-nd_*^2/8}.
\]
The union bound proves \eqref{eq:statistical-finite-class-bound};
summability and Borel--Cantelli prove eventual correctness.
For $k=2$, $P=1$ and $d_*=\delta/B$.
\end{proof}

The moment orders, evaluation points, clipping levels and separations
depend on the candidates. No uniform bound is asserted. For compact
smooth families,
\cref{prop:compactness-uniqueness-stability,cor:compact-minimum-distance-consistency}
instead apply to observations of complete paths.

\section{Marginal laws do not suffice}
\label{sec:hidden-Hermite-rank}

Consider
\[
 P(x,y)=x(x^2+y^2),\qquad Q(x,y)=x^3-3xy^2.
\]
Their real zero sets consist of one and three distinct lines, so no
orthogonal change of variables relates them.

\begin{proposition}[Equal marginals and different correlated laws]
\label{prop:same-marginal-different-replica}
For a standard Gaussian $G=(X,Y)$ in $\mathbb R^2$,
\begin{equation}
 P(G)\stackrel{\mathrm{law}}=Q(G).
 \label{eq:hidden-rank-same-marginal}
\end{equation}
For jointly Gaussian standard vectors $G_a=(X_a,Y_a)$ satisfying
$\E X_1X_2=\E Y_1Y_2=\rho$ and $\E X_aY_b=0$,
\begin{align}
 \E[P(G_1)P(G_2)]&=16\rho+8\rho^3,
 \label{eq:hidden-rank-P-covariance}\\
 \E[Q(G_1)Q(G_2)]&=24\rho^3.
 \label{eq:hidden-rank-Q-covariance}
\end{align}
Thus their two-point laws differ whenever $0<|\rho|<1$.
\end{proposition}

\begin{proof}
In polar coordinates, $G=(R\cos\Theta,R\sin\Theta)$ with $R$
independent of uniform $\Theta$ modulo $2\pi$, and
\[
 P(G)=R^3\cos\Theta,\qquad Q(G)=R^3\cos(3\Theta).
\]
The map $\Theta\mapsto3\Theta$ preserves Haar measure, proving
\eqref{eq:hidden-rank-same-marginal}.
For the probabilists' Hermite polynomials,
\begin{align}
 P(x,y)&=H_3(x)+H_1(x)H_2(y)+4H_1(x),
 \label{eq:hidden-rank-P-Hermite}\\
 Q(x,y)&=H_3(x)-3H_1(x)H_2(y).
 \label{eq:hidden-rank-Q-Hermite}
\end{align}
Use $\E[H_j(Z_1)H_k(Z_2)]=\mathbf1_{\{j=k\}}j!\rho^j$
and coordinate independence: the first-chaos variance is $16$ for
$P$, and the third-chaos variances are $3!+2!=8$ for $P$ and
$3!+9\cdot2!=24$ for $Q$.
\end{proof}

The Hermite rank, the smallest nonzero chaos order, is therefore $1$
for $P$ and $3$ for $Q$. Neither it nor the tensor orbit is determined
by a marginal law.

\begin{remark}[Covariance summability]
\label{cor:hidden-rank-memory}
Let $(G_n)_{n\in\mathbb Z}$ have independent coordinates, each a
centered stationary unit-variance Gaussian process with
$r(k)\sim c|k|^{-\alpha}$, $c>0$, $1/3<\alpha<1$.
Then $(P(G_n))$ and $(Q(G_n))$ have equal one-time laws, but
\[
 \sum_{k\in\mathbb Z}|\operatorname{Cov}(P(G_0),P(G_k))|=\infty,
 \qquad
 \sum_{k\in\mathbb Z}|\operatorname{Cov}(Q(G_0),Q(G_k))|<\infty.
\]
Their covariances are asymptotic to $16c|k|^{-\alpha}$ and
$24c^3|k|^{-3\alpha}$, respectively.
\end{remark}

\section{A dilation characterization of the Gaussian covariance kernel}
\label{app:gaussian-dilation}

The identity $K(\sqrt n\,u,\sqrt n\,v)=K(u,v)^n$ characterizes the
Gaussian covariance under a second-order assumption at the origin.

\begin{proposition}[A covariance kernel determined by a dilation identity]
\label{prop:canonical-gaussian-canvas}
Let $\widetilde K$ be a continuous real stationary covariance kernel
on $\mathbb R$, normalized by $\widetilde K(u,u)=1$.
Assume that $\kappa(s)=\widetilde K(s,0)$ is $C^2$ at zero,
$\kappa''(0)<0$, and, for an integer $n_0\ge2$ and $a>1$,
\[
 \widetilde K(au,av)=\widetilde K(u,v)^{n_0}.
\]
Then
\begin{equation}
 a=\sqrt{n_0},\qquad \widetilde K(u,v)=e^{-c(u-v)^2},\qquad
 c=-\tfrac12\kappa''(0)>0.
 \label{eq:canonical-Gaussian-conclusion}
\end{equation}
Rescaling gives \eqref{eq:universal-gaussian-kernel}; the normalization
$-\kappa''(0)=2$ gives $c=1$.
\end{proposition}

\begin{proof}
Stationarity and symmetry give $\widetilde K(u,v)=\kappa(u-v)$ with
$\kappa$ even. Differentiating $\kappa(as)=\kappa(s)^{n_0}$ twice
at zero yields $a^2\kappa''(0)=n_0\kappa''(0)$, hence
$a=\sqrt{n_0}$. Iteration gives
\[
 \kappa(s)=\kappa(a^{-j}s)^{n_0^j}>0
\]
for large $j$, since $\kappa(0)=1$. Thus $\varphi=-\log\kappa$ is
defined everywhere and $\varphi(r)=cr^2+o(r^2)$ at zero. Consequently,
\[
 \varphi(s)=n_0^j\varphi(n_0^{-j/2}s)=cs^2+o(1)
 \quad(j\to\infty),
\]
which proves \eqref{eq:canonical-Gaussian-conclusion}.
\end{proof}

Let $\widetilde\Psi:\mathbb R\to\mathscr G_{\widetilde K}$ have
dense linear span and Gram kernel $\widetilde K$. Since
\[
 \langle\widetilde\Psi(\sqrt n\,u),\widetilde\Psi(\sqrt n\,v)\rangle
 =\widetilde K(u,v)^n
 =\langle\widetilde\Psi(u)^{\otimes n},
             \widetilde\Psi(v)^{\otimes n}\rangle,
\]
define first on finite linear combinations
\begin{equation}
 T_n\left(\sum_i c_i\widetilde\Psi(\sqrt n\,u_i)\right)
       =\sum_i c_i\widetilde\Psi(u_i)^{\otimes n}.
 \label{eq:feature-tensor-dilation-map}
\end{equation}
Both squared norms equal
$\sum_{i,j}c_ic_j\widetilde K(u_i,u_j)^n$.
Thus $T_n$ is well-defined and isometric, and extends uniquely to the
closed spans. In
\eqref{eq:gaussian-kernel-fock-factorization}, $\Psi(u)$ is the
normalized exponential vector at $\sqrt2u$; this is the tensor-product
identity for exponential vectors in symmetric Fock space.

\section*{Use of AI tools}
AI tools assisted with manuscript editing, proof and reference checks,
and the reorganization of the unknown-time argument.

\begingroup
\small
\hbadness=2000

\endgroup

\end{document}